\documentclass[11pt,a4paper]{amsart}

\usepackage[utf8]{inputenc}

\usepackage{cmap}
\usepackage[T1]{fontenc}
\usepackage{amsmath,amssymb,amsthm,amsfonts}
\usepackage{mathtools}
\usepackage{microtype}
\usepackage{enumitem}
\usepackage[margin=1in]{geometry}
\usepackage{xcolor}
\usepackage[all]{xy}
\usepackage{hyperref}
\hypersetup{
  colorlinks=true,
  linkcolor=blue!50!black,
  citecolor=blue!50!black,
  urlcolor=blue!50!black,
}

\newtheorem{theorem}{Theorem}[section]
\newtheorem{proposition}[theorem]{Proposition}
\newtheorem{lemma}[theorem]{Lemma}
\newtheorem{corollary}[theorem]{Corollary}
\newtheorem{conjecture}[theorem]{Conjecture}
\theoremstyle{definition}
\newtheorem{definition}[theorem]{Definition}
\newtheorem{example}[theorem]{Example}

\theoremstyle{remark}
\newtheorem{remark}[theorem]{Remark}

\DeclareMathAlphabet{\pazocal}{OMS}{zplm}{m}{n}

\def\calF{\pazocal{F}}

\def\calH{\pazocal{H}}

\def\calO{\pazocal{O}}
\def\calP{\pazocal{P}}
\def\calQ{\pazocal{Q}}
\def\calR{\pazocal{R}}

\def\calX{\pazocal{X}}

\DeclareSymbolFont{symbolsC}{U}{pxsyc}{m}{n}
\SetSymbolFont{symbolsC}{bold}{U}{pxsyc}{bx}{n}
\DeclareFontSubstitution{U}{pxsyc}{m}{n}

\DeclareMathSymbol{\medcirc}{\mathbin}{symbolsC}{7}
\DeclareMathSymbol{\medbullet}{\mathbin}{symbolsC}{8}

\newcommand{\bk}{\Bbbk}
\newcommand{\AntiPL}{\mathsf{AntiPreLie}}
\newcommand{\AntiPLdual}{\mathsf{AntiPreLie}^{!}}

\newcommand{\Lie}{\mathsf{Lie}}
\newcommand{\Nov}{\mathsf{Nov}}

\newcommand{\Com}{\mathsf{Com}}

\newcommand{\Free}{\calF}
\newcommand{\sgn}{\mathrm{sgn}}
\newcommand{\id}{\mathrm{id}}

\DeclareMathOperator{\Hom}{Hom}

\DeclareMathOperator{\ad}{ad}
\DeclareMathOperator{\Alt}{Alt}
\DeclareMathOperator{\rl}{last}

\newcommand{\LXX}[3]{\ensuremath{
 \xygraph{
!{<0pt,0pt>;<4pt,0pt>:<0pt,-4pt>::}
!{(1,2)}="a"
!{(1,1)}*{\scriptstyle{\bullet}}="b"
!{(2,1)}*{\scriptscriptstyle {}}
!{(-1,-1)}*{\scriptstyle{\bullet}}="c"
!{(0,-1)}*{\scriptscriptstyle {}}
!{(3,-1)}="d"
!{(3,-2)}*{\scriptscriptstyle #3}
!{(-3,-3)}="e"
!{(-3,-4)}*{\scriptscriptstyle #1}
!{(1,-3)}="f"
!{(1,-4)}*{\scriptscriptstyle #2}
"a"-"b"
"b"-"c"
"b"-"d"
"c"-"e"
"c"-"f"
}
}}

\newcommand{\RXX}[3]{\ensuremath{
 \xygraph{
!{<0pt,0pt>;<4pt,0pt>:<0pt,-4pt>::}
!{(5,-4)}*{\scriptscriptstyle #3}
!{(1,2)}="a"
!{(1,1)}*{\scriptstyle{\bullet}}="b"
!{(1,-4)}*{\scriptscriptstyle #2}
!{(-1,-1)}="c"
!{(3,-1)}*{\scriptstyle{\bullet}}="d"
!{(5,-3)}="e"
!{(-1,-2)}*{\scriptscriptstyle #1}
!{(1,-3)}="f"
"a"-"b"
"b"-"c"
"b"-"d"
"d"-"e"
"d"-"f"
}
}}

\newcommand{\LZZ}[3]{\ensuremath{
 \xygraph{
!{<0pt,0pt>;<4pt,0pt>:<0pt,-4pt>::}
!{(1,2)}="a"
!{(1,1)}*{\scriptscriptstyle{}}="b"
!{(-1,-1)}*{\scriptscriptstyle{}}="c"
!{(3,-1)}="d"
!{(3,-2)}*{\scriptscriptstyle #3}
!{(-3,-3)}="e"
!{(-3,-4)}*{\scriptscriptstyle #1}
!{(1,-3)}="f"
!{(1,-4)}*{\scriptscriptstyle #2}
"a"-"b"
"b"-"c"
"b"-"d"
"c"-"e"
"c"-"f"
}
}}

\newcommand{\RZZ}[3]{\ensuremath{
 \xygraph{
!{<0pt,0pt>;<4pt,0pt>:<0pt,-4pt>::}
!{(5,-4)}*{\scriptscriptstyle #3}
!{(1,2)}="a"
!{(1,1)}*{\scriptscriptstyle{}}="b"
!{(1,-4)}*{\scriptscriptstyle #2}
!{(-1,-1)}="c"
!{(3,-1)}*{\scriptscriptstyle{}}="d"
!{(5,-3)}="e"
!{(-1,-2)}*{\scriptscriptstyle #1}
!{(1,-3)}="f"
"a"-"b"
"b"-"c"
"b"-"d"
"d"-"e"
"d"-"f"
}
}}

\newcommand{\LXZ}[3]{\ensuremath{
 \xygraph{
!{<0pt,0pt>;<4pt,0pt>:<0pt,-4pt>::}
!{(1,2)}="a"
!{(1,1)}*{\scriptstyle{\bullet}}="b"
!{(-1,-1)}*{\scriptscriptstyle{}}="c"
!{(3,-1)}="d"
!{(3,-2)}*{\scriptscriptstyle #3}
!{(-3,-3)}="e"
!{(-3,-4)}*{\scriptscriptstyle #1}
!{(1,-3)}="f"
!{(1,-4)}*{\scriptscriptstyle #2}
"a"-"b"
"b"-"c"
"b"-"d"
"c"-"e"
"c"-"f"
}
}}

\newcommand{\RXZ}[3]{\ensuremath{
 \xygraph{
!{<0pt,0pt>;<4pt,0pt>:<0pt,-4pt>::}
!{(5,-4)}*{\scriptscriptstyle #3}
!{(1,2)}="a"
!{(1,1)}*{\scriptstyle{\bullet}}="b"
!{(1,-4)}*{\scriptscriptstyle #2}
!{(-1,-1)}="c"
!{(3,-1)}*{\scriptscriptstyle{}}="d"
!{(5,-3)}="e"
!{(-1,-2)}*{\scriptscriptstyle #1}
!{(1,-3)}="f"
"a"-"b"
"b"-"c"
"b"-"d"
"d"-"e"
"d"-"f"
}
}}

\newcommand{\LZX}[3]{\ensuremath{
 \xygraph{
!{<0pt,0pt>;<4pt,0pt>:<0pt,-4pt>::}
!{(1,2)}="a"
!{(1,1)}*{\scriptscriptstyle{}}="b"
!{(-1,-1)}*{\scriptstyle{\bullet}}="c"
!{(3,-1)}="d"
!{(3,-2)}*{\scriptscriptstyle #3}
!{(-3,-3)}="e"
!{(-3,-4)}*{\scriptscriptstyle #1}
!{(1,-3)}="f"
!{(1,-4)}*{\scriptscriptstyle #2}
"a"-"b"
"b"-"c"
"b"-"d"
"c"-"e"
"c"-"f"
}
}}

\newcommand{\RZX}[3]{\ensuremath{
 \xygraph{
!{<0pt,0pt>;<4pt,0pt>:<0pt,-4pt>::}
!{(5,-4)}*{\scriptscriptstyle #3}
!{(1,2)}="a"
!{(1,1)}*{\scriptscriptstyle{}}="b"
!{(1,-4)}*{\scriptscriptstyle #2}
!{(-1,-1)}="c"
!{(3,-1)}*{\scriptstyle{\bullet}}="d"
!{(5,-3)}="e"
!{(-1,-2)}*{\scriptscriptstyle #1}
!{(1,-3)}="f"
"a"-"b"
"b"-"c"
"b"-"d"
"d"-"e"
"d"-"f"
}
}}

\newcommand{\Xsmall}[2]{\ensuremath{
 \xygraph{
!{<0pt,0pt>;<2.6pt,0pt>:<0pt,-2.6pt>::}
!{(0,1.5)}="a"
!{(0,0)}="b"
!{(0,0)}*{\scriptscriptstyle{}}
!{(-2,-2)}="c"
!{(2,-2)}="d"
!{(-2.5,-3.5)}*{\scriptscriptstyle #1}
!{(2.5,-3.5)}*{\scriptscriptstyle #2}
"a"-"b"
"b"-"c"
"b"-"d"
}
}}

\newcommand{\Zsmall}[2]{\ensuremath{
 \xygraph{
!{<0pt,0pt>;<2.6pt,0pt>:<0pt,-2.6pt>::}
!{(0,1.5)}="a"
!{(0,0)}="b"
!{(0,0)}*{\scriptstyle{\bullet}}
!{(-2,-2)}="c"
!{(2,-2)}="d"
!{(-2.5,-3.5)}*{\scriptscriptstyle #1}
!{(2.5,-3.5)}*{\scriptscriptstyle #2}
"a"-"b"
"b"-"c"
"b"-"d"
}}}

\newcommand{\LZZZ}[4]{\ensuremath{
 \xygraph{
!{<0pt,0pt>;<4pt,0pt>:<0pt,-4pt>::}
!{(1,2)}="a"
!{(1,1)}="b"
!{(-1,-1)}="c"
!{(3,-1)}="d"
!{(3,-2)}*{\scriptscriptstyle #4}
!{(-3,-3)}="e"
!{(1,-4)}*{\scriptscriptstyle #3}
!{(1,-3)}="f"
!{(-5,-5)}="g"
!{(-1,-5)}="h"
!{(-5,-6)}*{\scriptscriptstyle #1}
!{(-1,-6)}*{\scriptscriptstyle #2}
"a"-"b"
"b"-"c"
"b"-"d"
"c"-"e"
"c"-"f"
"e"-"g"
"e"-"h"
}}}

\newcommand{\LXZZ}[4]{\ensuremath{
 \xygraph{
!{<0pt,0pt>;<4pt,0pt>:<0pt,-4pt>::}
!{(1,2)}="a"
!{(1,1)}="b"
!{(1,1)}*{\scriptstyle{\bullet}}
!{(-1,-1)}="c"
!{(3,-1)}="d"
!{(3,-2)}*{\scriptscriptstyle #4}
!{(-3,-3)}="e"
!{(1,-4)}*{\scriptscriptstyle #3}
!{(1,-3)}="f"
!{(-5,-5)}="g"
!{(-1,-5)}="h"
!{(-5,-6)}*{\scriptscriptstyle #1}
!{(-1,-6)}*{\scriptscriptstyle #2}
"a"-"b"
"b"-"c"
"b"-"d"
"c"-"e"
"c"-"f"
"e"-"g"
"e"-"h"
}}}

\newcommand{\LXXZ}[4]{\ensuremath{
 \xygraph{
!{<0pt,0pt>;<4pt,0pt>:<0pt,-4pt>::}
!{(1,2)}="a"
!{(1,1)}="b"
!{(1,1)}*{\scriptstyle{\bullet}}
!{(-1,-1)}="c"
!{(-1,-1)}*{\scriptstyle{\bullet}}
!{(3,-1)}="d"
!{(3,-2)}*{\scriptscriptstyle #4}
!{(-3,-3)}="e"
!{(1,-4)}*{\scriptscriptstyle #3}
!{(1,-3)}="f"
!{(-5,-5)}="g"
!{(-1,-5)}="h"
!{(-5,-6)}*{\scriptscriptstyle #1}
!{(-1,-6)}*{\scriptscriptstyle #2}
"a"-"b"
"b"-"c"
"b"-"d"
"c"-"e"
"c"-"f"
"e"-"g"
"e"-"h"
}}}

\title[The operad of anti-pre-Lie algebras]
{The operad of anti-pre-Lie algebras and its Lie elements}

\author{Chengming Bai}
\address{Chern Institute of Mathematics \& LPMC, Nankai University, Tianjin 300071, China }
\email{baicm@nankai.edu.cn}

\author{Vladimir Dotsenko}
\address{Institute of Mathematics, Academy of Mathematics and Systems Science, Chinese Academy of Sciences, Beijing, 100190, China\newline
Institut de Recherche Mathématique Avancée, UMR7501, Université de Strasbourg et CNRS, Strasbourg 67000, France}
\email{vdotsenko@unistra.fr}

\author{Guilai Liu}
\address{Chern Institute of Mathematics \& LPMC, Nankai University, Tianjin 300071, China
    \newline
    Institut de Recherche Mathématique Avancée, UMR7501, Université de Strasbourg, Strasbourg 67000, France
}
\email{liugl@mail.nankai.edu.cn}

\author{Hongjie Wang}
\address{Chern Institute of Mathematics \& LPMC, Nankai University, Tianjin 300071, China}
\email{hongjiewang@mail.nankai.edu.cn}

\microtypesetup{expansion=false}

\input{glyphtounicode}
\begin{document}

\begin{abstract}
We initiate a systematic operadic approach to anti-pre-Lie algebras. We study the operad controlling anti-pre-Lie algebras, its relationship to the operad of Novikov algebras and its cyclic structure. We establish that the Lie bracket of an anti-pre-Lie algebra satisfies the standard identity of degree five. We also establish a quadratic maximality property of the
anti-pre-Lie operad with respect to this identity.
We then show that the anti-pre-Lie operad and its Koszul dual are not Koszul and give a full description of the latter. Finally, we identify the Lie elements in the Koszul dual operad with the operad of metabelian Lie algebras.
\end{abstract}

\maketitle

\setcounter{tocdepth}{1}
\tableofcontents

\section{Introduction}

An anti-pre-Lie algebra \cite{LiuBai} is a vector space $A$ equipped
with a bilinear operation $\circ$ satisfying the two identities
\begin{align}
  &x \circ (y \circ z) - y \circ (x \circ z) = [y, x] \circ z,
  \label{eq:AP1} \\
  &[x, y] \circ z + [y, z] \circ x + [z, x] \circ y = 0,\;\forall x,y,z\in A,
  \label{eq:AP2}
\end{align}
where $[-,-]$ is the commutator given by $[x,y]=x\circ y-y\circ x$
for all $x,y\in A$. An anti-pre-Lie algebra $(A,\circ)$ is
equivalently characterized as a Lie-admissible algebra such that
the negative left multiplication 
operators give
a representation $(-L_{\circ}, A)$ of the commutator Lie
algebra, in which $L_{\circ}:A\rightarrow\mathrm{End}(A)$ is given
by $L_{\circ}(x)y=x\circ y$ for all $x,y\in A$. This
characterization exhibits a natural similarity with pre-Lie
algebras \cite{Bai,Bur}, where it is the left multiplication
operators that form a representation of the commutator Lie algebra.

A number of algebraic results involving anti-pre-Lie algebras have emerged in recent years.
For instance, anti-pre-Lie bialgebras were introduced in \cite{TPA} as the structures
corresponding to Manin triples of Lie algebras with nondegenerate commutative $2$-cocycles. They also provide a necessary ingredient in
the bialgebra theories of transposed Poisson algebras and relative Poisson algebras \cite{RPA}.
Anti-dendriform algebras were introduced in \cite{adend} as a new type of splitting of associative algebras, in contrast
with the classical splitting provided by dendriform algebras, and their relationship with anti-pre-Lie algebras was also investigated. Both
anti-pre-Lie and anti-dendriform algebras play important roles in the new splitting theory of (transposed) Poisson algebras \cite{NSP},
which continues the line of work initiated by the pre-Poisson algebras of \cite{Aguiar}.
On the structural side, graded anti-pre-Lie algebra structures on the Witt and Virasoro algebras were investigated in \cite{BG1}, while
compatible root-graded anti-pre-Lie algebra structures on finite-dimensional complex simple Lie algebras were
studied in \cite{BG2}. Several natural generalizations of anti-pre-Lie algebras including anti-pre-Lie
superalgebras, compatible anti-pre-Lie algebras and anti-pre-Leibniz algebras have also been studied, see \cite{ChenLiuChen,Normatov,ZhaoLiu}.

The main goal of this paper is to bring methods and ideas of operad theory into the study of anti-pre-Lie algebras. We study the corresponding operad and its Koszul dual, and establish some of their basic properties.

In the first half of the paper, we explain how to define the operad $\AntiPL$ controlling anti-pre-Lie algebras, and present the result of computations of the first few dimensions of its components, as well as the structures of those components as $S_n$-modules. We also discuss the surjective morphism from the operad $\AntiPL$ to the operad of Novikov algebras, as well as cyclic operad structures on these two operads. We also study the commutator Lie algebras of anti-pre-Lie algebras or, in other words, the morphism of operads $\kappa\colon\Lie\longrightarrow\AntiPL$. We identify a rather remarkable property of anti-pre-Lie algebras that previously remained unnoticed: we show that the commutator $a\circ b-b\circ a$ in an anti-pre-Lie algebra always satisfies the so-called standard identity of degree five
  \[
\sum_{\sigma \in S_{4}} \sgn(\sigma)
   \bigl[x_{\sigma(1)}, [x_{\sigma(2)}, [x_{\sigma(3)},
   [x_{\sigma(4)}, x_{5}]]]\bigr] = 0.
  \]
As we explain below, a similar result also holds for the operad of Novikov algebras, but for a much simpler reason: the arity five component of the latter operad does not contain the isotypic component of the standard identity. By contrast, $\AntiPL(5)$ contains a submodule isomorphic to the submodule of $\Lie(5)$ generated by the standard identity of degree five, which makes our result particularly surprising.
We further establish a quadratic maximality property of $\AntiPL$.
More precisely, its defining relations cannot be properly weakened
within the class of binary quadratic operads on the same generating
space while retaining a commutator Lie bracket satisfying the
standard identity of degree five.
A famous conjecture in the theory of identities in Lie algebras \cite{MR4368864} is that the standard identity of degree five implies all identities that hold in the Lie algebra of vector fields on the affine line. This raises a natural question of what kind of Lie algebras arise as commutator algebras of anti-pre-Lie algebras, and suggests that anti-pre-Lie algebras might shed new light on identities of the Lie algebra of vector fields on the affine line.

In the second half of the paper, we study the Koszul dual operad $\AntiPLdual$. We give an explicit presentation of $\AntiPLdual$ by three quadratic identities, and show that $\dim\AntiPLdual(n)$
$=n+1$ for every $n\geq3$. In fact, we give two proofs of this result: one using the fact that the operad $\AntiPLdual$ admits a finite Gröbner basis, and another
that exhibits
a convenient basis compatible with the
symmetric group actions.
 Using this basis, we prove that the suboperad of $\AntiPLdual$ generated by the commutator is isomorphic to the operad of metabelian Lie algebras (Proposition~\ref{prop:dual-metabelian}); its arity-$n$ component has dimension $n-1$ for $n\geq2$. We also use the obtained dimension formula for the operad $\AntiPLdual$ to establish that the operads $\AntiPL$ and $\AntiPLdual$ are not Koszul. We study the cyclic structure of the operad $\AntiPLdual$, and the related notion of a quadratic dual anti-pre-Lie algebra, which we characterise in finite dimensions through representation theory. Finally, we show that the tensor product of a quadratic anti-pre-Lie algebra and a quadratic dual anti-pre-Lie algebra gives rise to a quadratic Lie algebra, and outline a proof of a general operadic result behind this statement.

The paper is organized as follows. In Section~\ref{sec:operad}, we place anti-pre-Lie algebras in the context of operad theory, and prove the abovementioned results about that operad. In Section \ref{sec:Koszul}, we undertake a study of the Koszul dual operad $\AntiPLdual$, proving the abovementioned results about it. In Section~\ref{sec:open}, we present several open questions for future work. Finally, in Appendix \ref{app:standard}, we give a rather technical computational result proving that in Lie algebras arising as commutator algebras of anti-pre-Lie algebras, the standard identity of degree five holds.

Throughout the paper, $\bk$ denotes an arbitrary  field of characteristic zero, and all vector spaces and tensor products are taken over $\bk$.
By an operad we mean a symmetric operad in the category of $\bk$-vector spaces. We denote by $S_n$ the symmetric group on $n$ letters and by $S^\lambda$ the irreducible $S_n$-module corresponding to a partition $\lambda$ of $n$. Thus, in particular, $S^{n}$ and $S^{1^n}$ are the trivial and sign representations, respectively.
 
\section{The anti-pre-Lie operad and its properties}\label{sec:operad}

In this section, we study generators and relations of the operad $\AntiPL$ governing anti-pre-Lie algebras, determine dimensions of components of $\AntiPL$ in low arities and explain the absence of the sign representation of $S_n$ in $\AntiPL(n)$ for $n\ge 3$.
We also study the relationship to the operad of Novikov algebras, the cyclic structure and the Lie elements in the anti-pre-Lie operad.
Finally, we establish a quadratic maximality property of $\AntiPL$
with respect to the standard identity of degree five.

We refer the reader to \cite{BremnerDotsenko,LodayVallette} for basics of operad theory as well as examples of computation of Gröbner bases for operads. For a reader with background in nonassociative ring theory, we simply indicate that any type of algebras $\mathsf{S}$ without constants in its signature gives rise to a collection of $S_n$-modules $\{\calO_\mathsf{S}(n)\}_{n\ge 1}$, where $\calO_\mathsf{S}(n)$ is the vector space of multilinear elements in the free $\mathsf{S}$-algebra on $n$ generators, on which $S_n$ acts by permuting the generators. Multilinear operations can be substituted into one another, and these substitution operations give rise to a structure on $\{\calO_\mathsf{S}(n)\}_{n\ge 1}$ which is axiomatised by the notion of a (symmetric) operad. It turns out that an operad structure can be viewed as a structure of a monoid in a certain monoidal category, which implies that a lot of intuition coming from working with associative algebras (monoids in the monoidal category of vector spaces) can be brought into the theory of operads.

In particular, operads can be presented by generators and relations: the collection of $S_n$-modules of structure operations $\{\calX_\mathsf{S}(n)\}_{n\ge 1}$ of $\mathsf{S}$-algebras provides the generators of the operad $\mathcal{O}_\mathsf{S}$, and the axioms of $\mathsf{S}$-algebras become the relations of that operad. Note that, unlike in the classical theory of varieties of algebras \cite{ZSSS}, we include symmetries of the structure operations in the signature, so that, for instance, commutativity is a property of a generator, not an extra identity.

\subsection{Generators and relations of the operad of anti-pre-Lie algebras}

In the case of anti-pre-Lie algebras, we have only one binary
structure operation without any symmetry, and so
this operad is a quotient of the free operad $\Free(E)$ generated
by a collection $E$ consisting of the single non-zero module $E(2)
= \bk[S_{2}]$, which we view as the vector space with a basis
$x_1x_2$, $x_2x_1$. Elements of $\Free(E)(n)$ are obtained as
iterated substitution of the binary operations. In particular,
$\Free(E)(3)$ has a basis consisting of the twelve elements
\begin{align*}
  L(\sigma) &:=L(\sigma(1),\sigma(2),\sigma(3)):= (x_{\sigma(1)} \circ x_{\sigma(2)}) \circ x_{\sigma(3)},
  \\
  R(\sigma) &:=R(\sigma(1),\sigma(2),\sigma(3)):= x_{\sigma(1)} \circ (x_{\sigma(2)} \circ x_{\sigma(3)})
\end{align*}
for $\sigma \in S_{3}$.
The operad governing anti-pre-Lie algebras is the quotient operad
\[
\AntiPL = \Free(E)/(R_{\AntiPL}),
\]
where $R_{\AntiPL} \subseteq \Free(E)(3)$ is the $S_{3}$-submodule generated by the expansions of \eqref{eq:AP1} and \eqref{eq:AP2}.
By \eqref{eq:AP1} and \eqref{eq:AP2}, we respectively have
 \[
L(1,2,3) + R(1,2,3) - L(2,1,3) - R(2,1,3) = 0,
\]
and  \[
L(1,2,3) - L(2,1,3) + L(2,3,1) - L(3,2,1) + L(3,1,2) - L(1,3,2) = 0.
\]
 By a simple direct computation, we have $\dim R_{\AntiPL} = 4$,
and hence $\dim \AntiPL(3) = 12 - 4 = 8$.

\subsection{Dimensions and symmetric group actions}

In fact, using computer calculations, we can determine some further dimensions of components of the operad $\AntiPL$.

\begin{proposition}
\label{prop:dimensions}
The dimensions of components of the operad $\AntiPL$
in arities $1\le n\le 7$ are given as follows:
\begin{center}
	\begin{tabular}{|c|c|c|c|c|c|c|c|}
		\hline
		$n$ & $1$ & $2$ & $3$ & $4$ & $5$ & $6$ & $7$ \\
		\hline
		$\dim \AntiPL(n)$
		& $1$ & $2$ & $8$ & $45$ & $304$ & $2283$ & $18940$ \\
		\hline
	\end{tabular}
\end{center}
Moreover, the $S_n$-module structures of those components are given as follows:
\begin{center}
	\begin{tabular}{|c|c|}
		\hline
		$n$ & $\AntiPL(n)$ \\
		\hline
		$1$ & $S^{1}$ \\
		\hline
		$2$ & $S^{2}\oplus S^{1^2}$ \\
		\hline
		$3$ & $2S^{3}\oplus 3S^{2,1}$ \\
		\hline
		$4$ & $4S^{4}\oplus 8S^{3,1}\oplus 4S^{2^2}
		\oplus 3S^{2,1^2}$ \\
		\hline
		$5$ & $9S^{5}\oplus 20S^{4,1}\oplus 19S^{3,2}
		\oplus 11S^{3,1^2}\oplus 10S^{2^2,1}
		\oplus S^{2,1^3}$ \\
		\hline
		$6$ & $18S^{6}\oplus 49S^{5,1}\oplus 60S^{4,2}
		\oplus 37S^{4,1^2}\oplus 27S^{3^2}
		\oplus 48S^{3,2,1}\oplus 7S^{3,1^3}
		\oplus 13S^{2^3}\oplus 8S^{2^2,1^2}$ \\
		\hline
		$7$ & $39S^{7}\oplus 112S^{6,1}\oplus 167S^{5,2}
		\oplus 100S^{5,1^2}\oplus 122S^{4,3}
		\oplus 176S^{4,2,1}\oplus 34S^{4,1^3}
		\oplus 79S^{3^2,1}$ \\
		& $\oplus 79S^{3,2^2}\oplus 58S^{3,2,1^2}
		\oplus 5S^{3,1^4}\oplus 26S^{2^3,1}
		\oplus 4S^{2^2,1^3}$ \\
		\hline
	\end{tabular}
\end{center}
 \end{proposition}

\begin{proof}
For $n=1$ and $n=2$ the statement is trivial. 
For $n=3$, the dimension is obtained by a direct computation by hand, as noted above.
Further computations are organised as follows. First, we determine the dimensions for $4\le n\le 7$ using the computer program \cite{DH} that computes the low arity part of the reduced Gröbner basis of an operad for a given ordering, and then calculates the dimensions of components by enumerating the normal forms with respect to that Gröbner basis.

To determine the $S_n$-module structure, we use two different methods of computation that produced the same result. First, we use the strategy that efficiently implements Clifton matrices that give irreducible representations of the symmetric group, and then compute  dimensions of consequences of the defining relations of $\AntiPL$ in multiplicity spaces of irreducible representations of $S_n$ for $3\le n\le 7$. We invite the reader to consult the survey \cite{BMP} for details of this approach. As a sanity check, we compute, using the \texttt{albert} program \cite{albert}, all dimensions of multihomogeneous components of total degree at most $7$ in the free algebras with $k\le 7$ generators; this gives us the $GL_k$-characters, which we then convert to $S_n$-characters using transition from the $m$-basis to the $s$-basis of the ring of symmetric functions in \texttt{SageMath} \cite{sage}.
\end{proof}

\subsection{Absence of the sign representation}

The data of Proposition \ref{prop:dimensions} suggest that for $n\ge 3$ the sign representation does not appear in the decomposition of $\AntiPL(n)$ into irreducible $S_n$-modules. We shall now prove it, giving a method of proof which is actually sufficiently general to handle many similar cases.

\begin{proposition}
\label{prop:nosgn}
The sign representation of $S_n$ does not occur in $\AntiPL(n)$ for
$n\ge 3$; equivalently, the complete alternation
\[
\sum_{\sigma\in S_n}\sgn(\sigma)\,\sigma\cdot m
\]
of every multilinear monomial $m$ in $n\ge3$ variables vanishes in $\AntiPL$.
\end{proposition}

\begin{proof}
By \cite{LiuBai},
every anti-pre-Lie algebra $(A,\circ)$ satisfies
\begin{equation}\label{eq:altI}
  x\circ[y,z]+y\circ[z,x]+z\circ[x,y]=0.
\end{equation}
The identities \eqref{eq:AP2} and \eqref{eq:altI} are the complete alternations of the two binary tree
shapes in arity three, so both copies of the sign representation of
$\Free(E)(3)$ vanish in the quotient.

Consider now $n=4$. To prove the required result, it is enough to show, for every monomial $m\in \Free(E)(4)$, that the coset in $\AntiPL(4)$ of its complete alternation vanishes. This monomial is necessarily of the form $m'\circ m''$. If the arity of $m'$ or the arity of $m''$ is equal to $3$, then the coset of the complete alternation of $m$ vanishes since such alternation may be done by first performing a complete alternation of each factor $m'$ and $m''$, and we have just established that for $n=3$ there is no sign component in $\AntiPL(3)$.

It remains to treat the complete alternation of the monomial $m=(x_1\circ x_2)\circ(x_3\circ x_4)$.
Applying \eqref{eq:AP2} with
$
x=x_1,\; y=x_2,\; z=x_3\circ x_4,
$
we obtain
\[
[x_1,x_2]\circ(x_3\circ x_4)
+
[x_2,x_3\circ x_4]\circ x_1
+
[x_3\circ x_4,x_1]\circ x_2
=0.
\]
Applying $\Alt_4$ to this identity, the first term gives
\[
\Alt_4\!\left([x_1,x_2]\circ(x_3\circ x_4)\right)
=
2\,\Alt_4m,
\]
since $[x_1,x_2] = x_1\circ x_2-x_2\circ x_1$
and transposing $x_1$ and $x_2$ changes the sign of the alternation.
On the other hand, after expanding the brackets in the remaining two
terms, every resulting monomial is of the form $m'\circ m''$ with the arity of $m'$ or the arity of $m''$ equal to $3$,
and their total alternations have already been shown to
vanish. This completes the proof for $n=4$.

Finally, suppose that $n\ge 5$. In this case, we may argue directly by induction on $n$, with the basis being the cases $n=3$ and $n=4$ which we already established. A multilinear monomial $m$ is of the form $m' \circ m''$, and either $m'$ or $m''$ is of arity between $3$ and $n-1$; alternating over the variables of the corresponding factor already gives $0$ by the
induction hypothesis. This completes the proof.
\end{proof}

\subsection{Relationship to the operad of Novikov algebras}

Recall that a Novikov algebra \cite{BaN,Os} is an algebra with one binary operation satisfying the two identities
\begin{align}
  &x\star(y\star z)-y\star(x\star z)
  = [x,y]\star z, \\
  &(x\star y)\star z=(x\star z)\star y.
\end{align}
The operad $\Nov$ controlling Novikov algebras is well studied \cite{DzhumadildaevNovikov,MR3241181}. In the paper \cite{LiuBai} that introduced anti-pre-Lie algebras, an algebraic structure called \emph{admissible Novikov algebras} was introduced; these are algebras with one binary operation satisfying \eqref{eq:AP1} and
\begin{equation}
2x\circ[y,z]=(x\circ y)\circ z-(x\circ z)\circ y.
\end{equation}
It is shown in \cite{LiuBai} that these identities imply the anti-pre-Lie identities, and that the category of admissible Novikov algebras is isomorphic to the category of Novikov algebras via a change of binary operation to $a\circ b-2b\circ a$.
Note that this kind of change of operation has been used in the theory of nonassociative rings since the work of Albert \cite{AlbertPAss}; in the context of operads, it became standard since \cite{MarklRemm}.

From the point of view of the present paper, these results can be formulated as follows.

\begin{proposition}\label{prop:APLNov}
There is a surjective map of operads
 \[
\AntiPL\twoheadrightarrow \Nov
 \]
sending the binary operation $a\circ b$ to $a\star b+2b\star a$.
\end{proposition}

\begin{proof}
It is established in \cite[Proposition~3.3]{LiuBai} that the identities of admissible Novikov algebras imply the anti-pre-Lie identities, and therefore there is a surjective map from the operad of anti-pre-Lie algebras to the operad of admissible Novikov algebras. It is furthermore established \cite[Propositions~3.5, 3.6]{LiuBai} that for an admissible Novikov algebra with the product $a,b\mapsto a\circ b$, the operation $a\star b:=a\circ b-2b\circ a$ defines a Novikov algebra, and conversely, for a Novikov algebra with the product $a,b\mapsto a\star b$, the operation $a\circ b:=a\star b+2b\star a$ defines an admissible Novikov algebra.
In fact, these two procedures are not literally inverse to one another, but their composition multiplies the generating operation by $-3$, which is harmless since we work over a field of zero characteristic.
This shows that the notion of admissible Novikov algebras corresponds to a different choice of the generator of the operad of Novikov algebras, and that on the level of operads we have a (clearly surjective) map of operads implementing all the stated results.
\end{proof}

\subsection{The cyclic structure of the anti-pre-Lie operad}

Recall \cite{MR1358617} that an operad $\calP$ is said to be cyclic if for each $n\ge 1$, the action of $S_n$ on $\calP(n)$ may be extended to an action of $S_{n+1}$ in such a way that the long cycles $\tau_{n+1}=(1,\ldots,n,(n+1))\in S_{n+1}$ satisfy the property
 \[
\tau_{m+n}(\alpha\circ_i\beta)=
\begin{cases}
\,\, \tau_{m+1}(\alpha)\circ_{i+1}\beta, \,\quad 1\le i< m,\\
\tau_{n+1}(\beta)\circ_1\tau_{m+1}(\alpha), \quad i=m
\end{cases}
 \]
for all $\alpha\in\calP(m)$, $\beta\in\calP(n)$, and $\tau_2(\id)=\id$.
In particular, for an operad generated by binary operations, the extended action of $S_3$ on the space of generators fully defines the cyclic structure.

One can also think of cyclic operads as of algebras for the monad of non-rooted trees, and via invariant symmetric bilinear forms. The latter approach, which will be particularly useful for us, can be formulated as follows. For a cyclic operad $\calP$ as above, we say that a symmetric bilinear form $B(-,-)$ on a $\calP$-algebra $A$ is \emph{invariant} if
 \[
 B(\alpha(a_1,\ldots,a_n),a_{n+1})=B(\tau_{n+1}(\alpha)(a_2,\ldots,a_{n+1}),a_{1})
 \]
for all $\alpha\in\calP(n)$ and all $a_1,\ldots,a_{n+1}\in A$. To convert this into an operadic definition, we let $\calX$ be a sequence of representations of symmetric groups for which $\calX(n)=0$ for $n\ne 2$ and $\calX(2)$ is the trivial $S_2$-module; we shall fix an element $B(-,-)$ spanning $\calX(2)$. An equivalent definition of a cyclic operad says that it is given by an extension of the action of $S_n$ on $\calP(n)$ to an action of $S_{n+1}$ for which the quotient $Q$ of the free right $\calP$-module $\calX\circ\calP$ by the right $\calP$-submodule generated by the relations
\begin{equation}\label{eq:cycinv}
 B(\alpha(x_1,\ldots,x_n),x_{n+1})=B(\tau_{n+1}(\alpha)(x_2,\ldots,x_{n+1}),x_1), \quad \alpha\in\calP(n),
\end{equation}
has the property that the map from $\calP(n)$ to $Q(n+1)$ given by the formula
 \[
\alpha(x_1,\ldots,x_n)\mapsto [B(\alpha(x_1,\ldots,x_n),x_{n+1})]
 \]
is an isomorphism of vector spaces. In practice, this means that Equation \eqref{eq:cycinv} is compatible with the relations of the operad $\calP$: substituting one of the defining relations of $\calP$ in $B(-,-)$ and using \eqref{eq:cycinv} does not produce elements of the form $B(r(x_1,\ldots,x_n),x_{n+1})$ where $r$ is not a consequence of the defining relations.

\begin{proposition}\label{prop:cyclic}
The operads $\AntiPL$ and $\Nov$ are cyclic, and the morphism of operads
 \[
\AntiPL\to\Nov
 \]
is compatible with cyclic structures.
\end{proposition}

\begin{proof}
To prove the cyclicity of the operad $\AntiPL$, we recall that in \cite[Definition~2.22]{LiuBai}, the notion of an invariant symmetric bilinear form on an anti-pre-Lie algebra is introduced via the identity
\begin{equation}\label{eq:APL-inv}
B(x\circ y,z)=B(y,x\circ z-z\circ x).
\end{equation}
On the binary generating space, this condition defines the action
$\tau_3(\mu)=-\mu^{\mathrm{op}}$ and
$\tau_3(\mu^{\mathrm{op}})=\mu-\mu^{\mathrm{op}}$, where
$\mu(x,y)=x\circ y$ and
$\mu^{\mathrm{op}}(x,y)=\mu(y,x)$. Together with the transposition
interchanging $\mu$ and $\mu^{\mathrm{op}}$, these maps satisfy the
relations of $S_3$. We verify compatibility with the defining relations.
The calculation also uses
\begin{equation}\label{eq:B-consequence}
B(x\circ y,z)=-B(x,z\circ y)
\end{equation}
which follows from \eqref{eq:APL-inv}. Write
\begin{align*}
\rho_1(x,y,z)&=x\circ(y\circ z)-y\circ(x\circ z)+[x,y]\circ z,\\
\rho_2(x,y,z)&=[x,y]\circ z+[y,z]\circ x+[z,x]\circ y.
\end{align*}
To isolate the argument $x$ in $B(x\circ(y\circ z)-y\circ(x\circ z)+[x,y]\circ z,t)$, we write
\begin{gather*}
B(x\circ(y\circ z),t)=-B(x,t\circ(y\circ z)),\\
B(y\circ(x\circ z),t)=B(x\circ z,[y,t])=-B(x,[y,t]\circ z),\\
B((x\circ y)\circ z,t)=-B(x\circ y,t\circ z)=B(x,(t\circ z)\circ y),\\
B((y\circ x)\circ z,t)=-B(y\circ x,t\circ z)=-B(x,[y,t\circ z]),
\end{gather*}
implying that
\begin{multline*}
B(x\circ(y\circ z)-y\circ(x\circ z)+[x,y]\circ z,t)=\\
-B(x,t\circ(y\circ z))+B(x,[y,t]\circ z)+B(x,(t\circ z)\circ y)+B(x,[y,t\circ z])=\\
-B(x,t\circ(y\circ z))+B(x,[y,t]\circ z)+B(x,y\circ(t\circ z)).
\end{multline*}
Thus
\[
B(\rho_1(x,y,z),t)=B(x,\rho_1(y,t,z)).
\]
Since $\rho_1$ is skew-symmetric in its first two inputs, this also
handles its second input. Isolating its third input gives
\begin{align*}
B(\rho_1(x,y,z),t)
 &=B\bigl(z,[y,[x,t]]-[x,[y,t]]+[[x,y],t]\bigr)=0,
\end{align*}
by the Jacobi identity, which follows from \eqref{eq:AP1} and
\eqref{eq:AP2}. For the second defining relation, the same invariance
formulas give
\begin{align*}
B(\rho_2(x,y,z),t)
 &=B\bigl(x,y\circ(t\circ z)-z\circ(t\circ y)+[[y,z],t]\bigr)\\
 &=B\bigl(x,\rho_1(y,t,z)-\rho_1(z,t,y)+\rho_2(y,z,t)\bigr)\\
 &=0.
\end{align*}
Its other inputs are handled by the alternation of $\rho_2$.
Consequently, moving any input of either defining relation to the output
produces a consequence of the defining relations. This proves cyclicity
of $\AntiPL$.

To prove the cyclicity of the operad $\Nov$, we recall \cite[Definition~3.7]{MR4610291} that a symmetric bilinear form on a Novikov algebra is said to be invariant if
\begin{equation}\label{eq:Nov-inv}
B(x\star y,z)=-B(y,x\star z+z\star x).
\end{equation}
In principle, one may also verify compatibility with the identities of Novikov algebras by a direct computation, but in this case we managed to find a conceptual interpretation of the cyclicity. It is well known \cite{MR1918188} that the Novikov operad is isomorphic to the suboperad generated by the operation $a\star b:=\mu(a,D(b))$ of the operad of commutative algebras equipped with a derivation $D$. This latter operad is cyclic, with the action of $\tau_2$ on $D$ being the sign, and the action of $\tau_3$ on the commutative product $\mu$ being the trivial action. Moreover, we have
\begin{gather*}
\tau_3(\mu\circ_1 D)=\mu\circ_2 D,\\
\tau_3(\mu\circ_2 D)=\tau_2(D)\circ_1(\tau_3(\mu))=-D\circ_1\mu=-\mu\circ_1D-\mu\circ_2D,
\end{gather*}
showing at once that the suboperad generated by $\mu\circ_2 D$ is a cyclic suboperad, and that the cyclic action translates into the invariance condition \eqref{eq:Nov-inv}.

Furthermore, \cite[Proposition~3.11]{LiuBai} proves that a symmetric bilinear form on an admissible Novikov algebra $A$ that is invariant (if we view $A$ as an anti-pre-Lie algebra) leads, under the isomorphism between the admissible Novikov operad and the Novikov operad, to Condition \eqref{eq:Nov-inv}. This translates to the property of the morphism $\AntiPL\to\Nov$ being compatible with the cyclic structures.
\end{proof}

\begin{remark}
At a first glance, this proposition seems to contradict \cite[Proposition~4.3]{MR1358617} which asserts that for a cyclic quadratic operad $\calP$, the invariance condition for a symmetric bilinear form reads
 \[
B(\alpha(x_1,x_2),x_3)=B(x_1,\alpha(x_2,x_3)).
 \]
This corresponds to an important terminological confusion in the literature on cyclic operads. The foundational paper \cite{MR1358617} of Getzler and Kapranov discusses under the name ``cyclic quadratic operads'' the class of operads $\calP$ generated by binary operations for which the extension of the action of $S_2$ on $\calP(2)$ to an action of $S_3$ extends $S^2$ to $S^3$ and $S^{1^2}$ to $S^{1^3}$. However other extensions may be possible for a cyclic operad; for instance, for an operad generated by the regular representation of $S_2$, we can extend the action of $S_2$ on $\bk S_2\cong S^2\oplus S^{1^2}$ to $S^{2,1}$. Thus, an operad may be cyclic and quadratic but not ``cyclic quadratic'', as noted already in \cite[Remark~15]{MarklRemm}. In particular, for the operad $\AntiPL$, the action of $S_3$ on $\AntiPL(2)$ is precisely via the module $S^{2,1}$.
\end{remark}

\subsection{Lie elements in the anti-pre-Lie operad}
\label{sec:lie-structure}

For every anti-pre-Lie algebra, the commutator $[x, y] = x \circ y - y \circ x$ is a Lie bracket. In other words, there is a map of operads
 \[
\kappa\colon \Lie\to\AntiPL
 \]
sending the Lie bracket to $\mu - \mu \cdot (12)$. Let us give an easy argument explaining why this map is not injective; this preliminary argument will shortly be sharpened as a consequence of a more intricate analysis.

\begin{proposition}\label{prop:notinj}
The morphism
 \[
\kappa\colon \Lie\to\AntiPL
 \]
is not injective in arity $6$: its kernel contains the element
 \[
\sum_{\sigma \in S_{5}} \sgn(\sigma)
  \bigl[x_{\sigma(1)}, [x_{\sigma(2)}, [x_{\sigma(3)},
  [x_{\sigma(4)}, [x_{\sigma(5)},x_6]]]]\bigr].
 \]
\end{proposition}

\begin{proof}
Recall \cite[\S 4.8.3]{MR4368864} that the $S_n$-module
structures of $\Lie(n)$ for small $n$ are given as follows:
\begin{center}
	\begin{tabular}{|c|c|}
		\hline
		$n$ & $\Lie(n)$ \\
		\hline
		$1$ & $S^{1}$ \\
		\hline
		$2$ & $S^{1^2}$ \\
		\hline
		$3$ & $S^{2,1}$ \\
		\hline
		$4$ & $S^{3,1}\oplus S^{2,1^2}$ \\
		\hline
		$5$ & $S^{4,1}\oplus S^{3,2}\oplus S^{3,1^2}
		\oplus S^{2^2,1}\oplus S^{2,1^3}$ \\
		\hline
		$6$ & $S^{5,1}\oplus S^{4,2}\oplus 2S^{4,1^2}
		\oplus S^{3^2}\oplus 3S^{3,2,1}
		\oplus S^{3,1^3}\oplus 2S^{2^2,1^2}
		\oplus S^{2,1^4}$ \\
		\hline
	\end{tabular}
\end{center}
Comparing these decompositions with those of Proposition \ref{prop:dimensions}, we note that $\Lie(6)$ contains one copy of the irreducible module $S^{2,1^4}$, while $\AntiPL(6)$ does not contain any copies of this module. It is known \cite{MR4368864} that for every $n\geq3$, the $S_n$-module $\Lie(n)$ contains only one copy of $S^{2,1^{n-2}}$, and that copy is generated by the so-called standard identity
 \[
\sum_{\sigma\in S_{n-1}}\mathrm{sgn}(\sigma)\ad_{x_{\sigma(1)}}\ad_{x_{\sigma(2)}}\cdots \ad_{x_{\sigma(n-1)}}(x_n) .
 \]
This completes the proof.
\end{proof}

Since under the composite morphism
 $
\Lie\to\AntiPL\to\Nov
 $
the Lie bracket is sent to
 \[
a\star b+2b\star a-(b\star a+2a\star b)=-(a\star b-b\star a),
 \]
our result shows that the map of operads $\Lie\to\Nov$ is not injective. This is well known, and in fact, a similar argument allows for a stronger result.

\begin{proposition}\label{prop:notinjNov}
The morphism
 \[
\kappa_{\Nov}\colon \Lie\to\Nov
 \]
sending the Lie bracket to $a\star b-b\star a$ is not injective
in arity $5$: its kernel contains the element
 \[
\sum_{\sigma \in S_{4}} \sgn(\sigma)
  \bigl[x_{\sigma(1)}, [x_{\sigma(2)}, [x_{\sigma(3)},
  [x_{\sigma(4)},x_5]]]\bigr].
 \]
\end{proposition}

\begin{proof}
Indeed, for the Novikov operad, we have an $S_5$-module isomorphism
 \[
\Nov(5)\cong 5S^{5}\oplus 7S^{4,1}\oplus 4S^{3,2}\oplus 2S^{3,1^2}\oplus S^{2^2,1},
 \]
see \cite{MR3241181}, and in particular, the arity five component does not contain the module $S^{2,1^3}$ which is present in $\Lie(5)$. This completes the proof.
\end{proof}

Of course, this result is essentially classical: it is known that the Lie algebra $W_1$ of vector fields on the line satisfies the standard identity of degree five \cite{MR4368864}, and its Lie bracket is obtained from the Novikov product
 \[
f(x)\partial\star g(x)\partial = f(x)g'(x)\partial ,
 \]
which does not satisfy any additional Novikov identities
\cite{MR4549103,MR2884985}, meaning that the
standard identity of degree five holds for the Lie bracket on any
Novikov algebra.

Our study of Lie elements in the anti-pre-Lie operad led us to the following result.

\begin{theorem}\label{thm:lie-kernel}
The morphism $\kappa\colon \Lie\longrightarrow\AntiPL$ is injective in arities $n\leq4$. In arity $5$, the kernel of the morphism $\kappa$ is generated by the standard identity of degree $5$.
\end{theorem}

\begin{proof}
If the morphism $\kappa$ were not injective in arities less than five, the same would be true for the morphism from the Lie operad to the Novikov operad, which, as we saw above, would mean that the Lie algebra $W_1$ of vector fields on the line would satisfy some Lie identities of degree less than five. Identities of that Lie algebra are well studied, and it is proved by Molev \cite[Theorem~1]{Molev} that if we denote by $\calO_{W_1}$ the quotient of the operad $\Lie$ by the ideal of identities of $W_1$, then there is an isomorphism of $S_n$-modules
 \[
\calO_{W_1}(n)\cong\calH_n^{n-1},
 \]
where $\calH_n$ is the vector space of $S_n$-harmonic polynomials, and $\calH_n^k$, $k\ge 0$ denotes its subspace of homogeneous elements of degree $k$. It is well known that we have an isomorphism of graded vector spaces
 \[
\calH_n\cong \bk[x_1,\ldots,x_n]/(\bk[x_1,\ldots,x_n]_+^{S_n}),
 \]
           which immediately leads to the dimensions $1,1,2,6,20,71,259, 961$ \ldots of the components $\calO_{W_1}$
$(n)$. In particular, for $n\le 4$ we have $\dim\calO_{W_1}(n)=\dim\Lie(n)$, so the Lie algebra $W_1$ does not satisfy any nontrivial identities of arities less than five.
The fact that $\dim\calO_{W_1}(5)=20$ and $\dim\Lie(5)=24$ is an independent sanity check of Proposition \ref{prop:notinjNov}, and at the same time an upper bound
 \(
 \dim\ker\kappa(5)\le 4,
 \)
since
 \[
\ker(\Lie\to\AntiPL)\subset\ker(\Lie\to\Nov).
 \]

However, the preceding representation-theoretic argument does not establish that the standard identity of degree five holds in every commutator Lie algebra of an anti-pre-Lie algebra, since we saw in Proposition \ref{prop:dimensions} that the component $\AntiPL(5)$ contains one copy of the module $S^{2,1^3}$.
For completeness, we give a direct proof of the standard identity of degree $5$ in Appendix \ref{app:standard}.
It remains to recall that, as we already mentioned above, the standard identity generates the submodule $S^{2,1^3}$ of $\Lie(5)$, which is four-dimensional, proving that our upper bound on $\dim\ker\kappa(5)$ is actually sharp, and the kernel of $\kappa$ in arity five is generated by the standard identity.
\end{proof}

\begin{remark} \label{rem:char3}
The characteristic-zero assumption turns out to be crucial for this theorem: it turns out that over a field of characteristic three, the Lie bracket in the anti-pre-Lie operad does not satisfy any additional identities. To see that, let us take the quotient of the anti-pre-Lie operad by the ideal generated by $a\circ b+b\circ a$, or, equivalently, let us consider anti-pre-Lie algebras that are additionally skew-commutative. Then the identities of anti-pre-Lie algebras become
\begin{gather*}
a\circ(b\circ c)-b\circ(a\circ c)-(a\circ b)\circ c+3(a\circ b)\circ c=0,\\
2(a\circ b)\circ c+2(b\circ c)\circ a+2(c\circ a)\circ b=0,
\end{gather*}
and over a field of characteristic three both of these identities give the Jacobi identity. Thus, over a field of characteristic three there is a surjective map
 \[
\AntiPL\twoheadrightarrow\Lie.
 \]
The composite with $\kappa$ sends the Lie generator to twice itself.
Since $2$ is invertible in characteristic three, this composite is an
automorphism of $\Lie$, so $\kappa$ is injective. Moreover, one can check, for instance using the \texttt{albert} program~\cite{albert}, that over a field of characteristic three, the dimension of the component $\AntiPL(5)$ is $308$ and not $304$.
The particularity of characteristic three is also visible in the proof of Proposition \ref{prop:APLNov} where we see that the operad of admissible Novikov algebras is not isomorphic to the operad of Novikov algebras in that characteristic.
 \end{remark}

\subsection{Quadratic maximality}
\label{subsec:apl-quadratic-maximality}

From Theorem~\ref{thm:lie-kernel}, we know that the standard identity of
degree five persists under the proper enlargement from admissible
Novikov algebras to anti-pre-Lie algebras.
We now show that the anti-pre-Lie identities cannot be further
weakened, within the class of binary quadratic Lie-admissible
operads generated by $E$, while retaining this property.

\begin{proposition}
  \label{prop:apl-quadratic-maximality}
Let $R\subseteq R_{\AntiPL}$ be an $S_3$-submodule and $\mathcal{P}$ be a binary quadratic operad given by
 $\mathcal P=\mathcal F(E)/(R)$.
  Suppose that the commutator morphism
\[
\Lie\longrightarrow\mathcal P,
\qquad
[x,y]\longmapsto x\circ y-y\circ x,
\] exists and
annihilates the standard Lie polynomial of degree five.
Then $R=R_{\AntiPL}$, and hence the canonical morphism $
\mathcal P\twoheadrightarrow\AntiPL$
 is an isomorphism.
\end{proposition}

\begin{proof}
Let us examine possible candidates $R\subseteq R_{\AntiPL}$. The existence of the commutator morphism implies that the element 
 \[
\operatorname{Jac}(x,y,z)=[[x,y],z]+[[y,z],x]+[[z,x],y]  
 \]
is contained in $R$. As an $S_3$-module, $R_{\AntiPL}\cong S^{2,1}\oplus 2S^{1^3}$, and $R_{\AntiPL}/\bk\operatorname{Jac}\cong S^{2,1}\oplus S^{1^3}$ is multiplicity-free. Thus, there are four possible candidates for $R$, corresponding to four choices of a subset of irreducible constituents of $R_{\AntiPL}/\bk\operatorname{Jac}$. Including both constituents would mean $R=R_{\AntiPL}$, so it remains to rule out three other choices. 

The isotypic component of $S^{1^3}$ in $R_{\AntiPL}$ is spanned by $\operatorname{Jac}$ and the element $\rho_2(x,y,z)$ from the proof of Proposition \ref{prop:cyclic}. In the notation of that proof, the isotypic component of $S^{2,1}$ is generated by the element 
  \[
\rho_{\mathrm{st}}(x,y,z)=3\rho_1(x,y,z)-2\rho_2(x,y,z)
  +\operatorname{Jac}(x,y,z).    
  \]
Indeed, this element $\rho_{\mathrm{st}}$ is skew-symmetric in its first two arguments and has zero cyclic sum in the free operad; it is nonzero, since the coefficient of $x\circ(y\circ z)$ in its expansion is $2$.

Note that if we additionally impose the anti-commutativity relation $a\circ b=-b\circ a$, the elements $\rho_2$ and $\operatorname{Jac}$ become proportional; this means that for $R/\bk\operatorname{Jac}=0$ or $R/\bk\operatorname{Jac}\cong S^{1^3}$, the quotient of $\mathcal P$ by the ideal generated by $a\circ b+b\circ a$ is the operad $\Lie$. Thus, in these two cases the standard Lie identity cannot hold for the Lie bracket in $\mathcal P$. 

The case $R/\bk\operatorname{Jac}=S^{2,1}$ is much more intricate, since, as we saw in Remark \ref{rem:char3}, over a field of characteristic different from three combining $\rho_{\mathrm{st}}$ with anti-commutativity forces the Lie bracket to be two-step nilpotent, and the above argument cannot work. To prove that the Lie bracket in $\mathcal P_{\mathrm{st}}=\mathcal F(E)/(\operatorname{Jac},\rho_{\mathrm{st}})$ does not satisfy the standard identity of degree five, we construct a linear functional $\Lambda$ on $\mathcal F(E)(5)$ that vanishes on
  $(\operatorname{Jac},\rho_{\mathrm{st}})(5)$ but is nonzero
  on the commutator expansion of the standard Lie polynomial
  of degree five.
  
  In the following table, juxtaposition denotes the binary
  operation, and each row specifies a planar binary tree $T$
  with five ordered leaves.
  For $1\le j\le5$, let $T_j$ be obtained by placing $x_5$
  in the $j$-th leaf and $x_1,x_2,x_3,x_4$, in this order,
  in the remaining leaves.
  The entry in column $j$ is $\Lambda(T_j)$.
  Every multilinear monomial has a unique expression
  $\sigma\cdot T_j$, where $\sigma\in S_4$ permutes
  $x_1,x_2,x_3,x_4$ and fixes $x_5$.
  We extend the table by
  \[
  \Lambda(\sigma\cdot T_j)
  =\operatorname{sgn}(\sigma)\Lambda(T_j)
  \]
  and then extend linearly to $\mathcal F(E)(5)$.
  
  \begin{center}
    \renewcommand{\arraystretch}{1.15}
    \begin{tabular}{c|rrrrr}
      $T(a,b,c,d,e)$ & $j=1$ & $j=2$ & $j=3$ & $j=4$ & $j=5$ \\\hline
      $a(b(c(de)))$ & $3$ & $-17$ & $-7$ & $17$ & $0$ \\
      $a(b((cd)e))$ & $3$ & $-17$ & $-8$ & $0$ & $-16$ \\
      $a((bc)(de))$ & $-9$ & $-6$ & $0$ & $-15$ & $0$ \\
      $a((b(cd))e)$ & $20$ & $26$ & $14$ & $0$ & $8$ \\
      $a(((bc)d)e)$ & $20$ & $4$ & $0$ & $8$ & $8$ \\
      $(ab)(c(de))$ & $12$ & $0$ & $-27$ & $-27$ & $0$ \\
      $(ab)((cd)e)$ & $12$ & $0$ & $0$ & $0$ & $0$ \\
      $(a(bc))(de)$ & $-18$ & $-18$ & $0$ & $0$ & $0$ \\
      $((ab)c)(de)$ & $0$ & $0$ & $0$ & $0$ & $0$ \\
      $(a(b(cd)))e$ & $12$ & $24$ & $-12$ & $0$ & $0$ \\
      $(a((bc)d))e$ & $12$ & $12$ & $0$ & $24$ & $0$ \\
      $((ab)(cd))e$ & $36$ & $0$ & $36$ & $0$ & $0$ \\
      $((a(bc))d)e$ & $0$ & $0$ & $0$ & $0$ & $24$ \\
      $(((ab)c)d)e$ & $0$ & $0$ & $0$ & $0$ & $24$
    \end{tabular}
  \end{center}
  
  We verify that $\Lambda$ annihilates the required ideal component.
  For a multilinear polynomial $p$ in $n$ variables, put
  \[
  \begin{aligned}
    \operatorname{Ins}_n(p)
    ={}&\{x_{n+1}\circ p,\ p\circ x_{n+1}\}\\
    &{}\cup
    \{p(x_1,\ldots,x_{i-1},x_i\circ x_{n+1},
    x_{i+1},\ldots,x_n):1\le i\le n\}.
  \end{aligned}
  \]
  The space $(\operatorname{Jac},\rho_{\mathrm{st}})(5)$
  is spanned by the $S_5$-translates of polynomials obtained
  by applying $\operatorname{Ins}_3$ and then
  $\operatorname{Ins}_4$ to an $S_3$-translate of
  $\operatorname{Jac}$ or $\rho_{\mathrm{st}}$.
  Indeed, this ideal component is spanned by trees with one
  ternary vertex labelled by an $S_3$-translate of a defining
  relation and two binary vertices labelled by the generating
  operation.
  Such trees are obtained by two successive binary insertions,
  with permutations accounting for both orders of the inputs.
  By equivariance, permutations introduced between the insertions
  can be moved to the final $S_5$-action.
  
Since $\Lambda$ is alternating under $S_4$, it is enough
to check representatives of $S_4\backslash S_5$, where
\[
S_4\backslash S_5
=
\{S_4g:g\in S_5\},
\qquad
S_4g=\{hg:h\in S_4\}.
\]
  We take $\pi_k=(k\,5)$ for $1\le k\le4$ and
  $\pi_5=\mathrm{id}$.
  For $g\in S_5$, put $k=g^{-1}(5)$ and $h=g\pi_k$.
  Then $h\in S_4$ and $g=h\pi_k$.
  Thus, for every $p\in\mathcal F(E)(5)$,
  \[
  \Lambda(g\cdot p)
  =
  \operatorname{sgn}(h)\Lambda(\pi_k\cdot p).
  \]
  For each $r\in\{\operatorname{Jac},\rho_{\mathrm{st}}\}$
  and each $\sigma\in S_3$, choose
  \[
  p_4\in\operatorname{Ins}_3(\sigma\cdot r),
  \qquad
  p\in\operatorname{Ins}_4(p_4).
  \]
  An exact computation using the table gives
  \[
  \Lambda(\pi_k\cdot p)=0
  \qquad (1\le k\le5)
  \]
  for all such choices.
  These are $2\cdot6\cdot5\cdot6\cdot5=1800$ integer checks,
  obtained by expanding the polynomials and evaluating their
  monomials using the table.
  They establish that $\Lambda$ vanishes on
  $(\operatorname{Jac},\rho_{\mathrm{st}})(5)$.
  
  Finally, expanding the nested commutator gives
  \[
  \Lambda\bigl([x_1,[x_2,[x_3,[x_4,x_5]]]]\bigr)=9.
  \]
  By the alternating property of $\Lambda$, we obtain
  \[
  \Lambda\left(
  \sum_{\sigma\in S_4}\operatorname{sgn}(\sigma)
  [x_{\sigma(1)},[x_{\sigma(2)},
  [x_{\sigma(3)},[x_{\sigma(4)},x_5]]]]
  \right)
  =24\cdot9=216\ne0,
  \]
 completing the proof.
\end{proof}

\section{The Koszul dual of the anti-pre-Lie operad and its properties}\label{sec:Koszul}

    In this section, we give an explicit presentation and a dimension formula for the Koszul dual operad $\AntiPLdual$ of $\AntiPL$.
We conclude that both $\AntiPL$ and $\AntiPLdual$ are not Koszul.
We also identify the suboperad of $\AntiPLdual$ generated by the commutator with the operad of metabelian Lie algebras.
Then we introduce the notion of quadratic dual anti-pre-Lie algebras,
characterize them in finite dimensions through representation theory and relate them to quadratic Lie algebras.

 \subsection{Recollections on quadratic operads and Koszul duality}
 \label{sec:preliminaries}

 We recall only what is needed below, referring to \cite[Chapters 5--7]{LodayVallette} for details. Let $E=\bk[S_2]$, and let $\calP=\Free(E)/(R)$ be an operad with $R\subset \Free(E)(3)$; the fact that the axioms relating the defining operations are ternary 
 identities
 means that our operad $\calP$ is \emph{quadratic} (since a ternary identity is a combination of compositions of exactly two structure operations). The version of Koszul duality used below applies to quadratic operads.

Let $E^*=\operatorname{Hom}_{\bk}(E,\bk)$ be the linear dual of $E$.
For $f\in E^*$ and $e\in E$, we write
$\langle f,e\rangle=\langle e,f\rangle=f(e)$
for the natural pairing between $E^*$ and $E$.
The $S_2$-action on $E$ induces the contragredient action on $E^*$,
defined by
\begin{equation}
    \langle \sigma(f),e\rangle
    =\langle f,\sigma^{-1}(e)\rangle,
    \quad \sigma\in S_2.
\end{equation}

Let $\sgn_2$ denote the one-dimensional sign representation of $S_2$,
on which $\sigma$ acts by multiplication by $\sgn(\sigma)$.
The \emph{sign-twisted dual} of $E$ is defined by
$E^\vee=E^*\otimes\sgn_2$.
Thus, as a vector space, $E^\vee$ may be identified with $E^*$,
but its $S_2$-action is twisted by sign:
\[
\sigma\cdot_{E^\vee}f
=\sgn(\sigma)\,\sigma\cdot_{E^*}f.
\]
In particular, the nontrivial transposition $(12)$ acts on $E^\vee$ as
the negative of its action on $E^*$. This sign-twisted dual provides the
space of binary generators for the Koszul dual operad.
There is a canonical non-degenerate pairing
$\langle-,-\rangle\colon
\Free(E)(3)\otimes\Free(E^\vee)(3)\to\bk$, described as follows. Recall
that $\Free(E)(3)$ has a basis $\{L(\sigma),R(\sigma)\}_{\sigma\in S_3}$; we shall denote by $L'(\sigma),R'(\sigma)$
the analogous basis of $\Free(E^\vee)(3)$. Then the pairing
    $\langle-,-\rangle\colon
    \Free(E)(3)\otimes\Free(E^\vee)(3)\to\bk$
    is given by
    \[
    \langle L(\sigma),L'(\tau)\rangle=\sgn(\sigma)\delta_{\sigma,\tau},
\;
    \langle R(\sigma),R'(\tau)\rangle=-\sgn(\sigma)\delta_{\sigma,\tau},
\;
    \langle L(\sigma),R'(\tau)\rangle= \langle R(\sigma),L'(\tau)\rangle=0.
    \]
This is precisely the pairing of \cite[Lemma 7.6.3]{LodayVallette} for
$E=\bk[S_2]$, transported to the present labelling of the twelve
monomials. The sign discrepancy between the $L$- and $R$-families is
the usual one between the two partial compositions
$\circ_{\mathrm{I}}$ and $\circ_{\mathrm{II}}$, conceptually explained by suitable
operadic suspensions. This pairing is non-degenerate and $S_3$-sign-invariant:
\[
\langle g r,g q\rangle=\sgn(g)\langle r,q\rangle,
\quad g\in S_3.
\]
In particular, the orthogonal complement of an $S_3$-submodule is again
an $S_3$-submodule.
The \emph{Koszul dual} of the operad $\calP$ is, by definition,
\[
\calP^!=\Free(E^\vee)/(R^\perp),
\]
where $R^\perp=\{q\in\Free(E^\vee)(3)\mid
\langle r,q\rangle=0\text{ for all }r\in R\}$ is the orthogonal complement of $R$ with respect to
this pairing \cite[\S7.6]{LodayVallette}.

For comparison with \cite{LodayVallette}, the following table gives the
dictionary between our notation and the permutoassociahedron numbering
$u_1,\dots,u_{12}$ used in \cite[\S7.6]{LodayVallette}, together with
the diagonal signs $\epsilon_i=\langle u_i,v_i\rangle$, where
$v_1,\dots,v_{12}$ denote the corresponding basis of
$\Free(E^\vee)(3)$. In this numbering, the above pairing
takes the form $\langle u_i,v_j\rangle=(-1)^{i+1}\delta_{ij}$.
The $L$-type entries occupy the positions
$i\equiv0,1\pmod 4$, while the $R$-type entries occupy the positions
$i\equiv2,3\pmod 4$.

\begin{center}
	\renewcommand{\arraystretch}{1.15}
	\begin{tabular}{c|c|c|c@{\quad}|@{\quad}c|c|c|c}
		$i$ & monomial & our notation & $\epsilon_i$
		& $i$ & monomial & our notation & $\epsilon_i$ \\
		\hline
		$1$
		& $(x_1\circ x_2)\circ x_3$
		& $L(1,2,3)$ & $+1$
		& $7$
		& $x_3\circ(x_2\circ x_1)$
		& $R(3,2,1)$ & $+1$ \\
		
		$2$
		& $x_1\circ(x_2\circ x_3)$
		& $R(1,2,3)$ & $-1$
		& $8$
		& $(x_3\circ x_2)\circ x_1$
		& $L(3,2,1)$ & $-1$ \\
		
		$3$
		& $x_1\circ(x_3\circ x_2)$
		& $R(1,3,2)$ & $+1$
		& $9$
		& $(x_2\circ x_3)\circ x_1$
		& $L(2,3,1)$ & $+1$ \\
		
		$4$
		& $(x_1\circ x_3)\circ x_2$
		& $L(1,3,2)$ & $-1$
		& $10$
		& $x_2\circ(x_3\circ x_1)$
		& $R(2,3,1)$ & $-1$ \\
		
		$5$
		& $(x_3\circ x_1)\circ x_2$
		& $L(3,1,2)$ & $+1$
		& $11$
		& $x_2\circ(x_1\circ x_3)$
		& $R(2,1,3)$ & $+1$ \\
		
		$6$
		& $x_3\circ(x_1\circ x_2)$
		& $R(3,1,2)$ & $-1$
		& $12$
		& $(x_2\circ x_1)\circ x_3$
		& $L(2,1,3)$ & $-1$
	\end{tabular}
\end{center}

The Koszul duality theory is particularly powerful for a given operad $\calP$ if that operad possesses
the Koszul property, amounting to vanishing of certain homology groups. The cheapest way to \emph{disprove} the Koszul property
comes from the following generating series criterion of Ginzburg and Kapranov \cite{GinzburgKapranov}, see also \cite[Theorem 7.5.1]{LodayVallette}. Let $f_{\calP}(t)=\sum_{n\geq1}\dim\calP(n)\,t^n/n!$ denote the exponential generating series. If $\calP$ is Koszul, then
\begin{equation}\label{eq:GKtest}
    f_{\calP^{!}}\bigl(-f_{\calP}(-t)\bigr) = t.
    \end{equation}

\subsection{The Koszul dual \texorpdfstring{$\AntiPLdual$}{AntiPL!}}
\label{sec:koszul-dual}

Let us apply the general recipe recalled in the previous subsection to determine the Koszul dual operad of the operad
$\AntiPL$.

\begin{theorem}
\label{thm:koszul-dual}
The Koszul dual of $\AntiPL$ is the operad $\AntiPLdual$ generated by one binary operation $\star$ subject to the identities
    \begin{align}
        (a_1\star a_2)\star a_3
        &= (a_2\star a_1)\star a_3, \label{eq:DAPLI}\\
        a_1\star(a_2\star a_3)
        &= a_2\star(a_1\star a_3), \label{eq:DAPLII}\\
        (a_1\star a_2)\star a_3+a_1\star(a_2\star a_3)
        &= (a_2\star a_3)\star a_1+a_2\star(a_3\star a_1). \label{eq:DAPLIII}
    \end{align}
\end{theorem}

\begin{proof}
To lighten the notation, we drop the primes below and write $L(\sigma),R(\sigma)$ for the basis
elements of both $\Free(E)(3)$ and $\Free(E^{\vee})(3)$.

In either of $\Free(E)(3)$ and $\Free(E^{\vee})(3)$ put
\[
  J(\sigma) \;=\; L(\sigma) + R(\sigma), \quad \sigma \in S_{3},
\]
and let $\tau = (1\,2)$ and $c = (1\,2\,3)$ act on $S_{3}$ on the right,
by relabeling inputs,
 that is, if
    $\sigma=(i,j,k)$
 is written in one-line notation, then
    $ \sigma\tau=(j,i,k)$ and $\sigma c=(j,k,i) $.
 In this notation the two relation spaces are
generated as follows. On the $\AntiPL$ side, Relation \eqref{eq:AP1} is $J(\sigma) - J(\sigma\tau)$ and
Relation \eqref{eq:AP2} is $\sum_{\sigma}\sgn(\sigma)L(\sigma)$; on the $\AntiPLdual$
side, \eqref{eq:DAPLI} gives the vectors $L(\sigma) - L(\sigma\tau)$, \eqref{eq:DAPLII} gives
$R(\sigma) - R(\sigma\tau)$, and \eqref{eq:DAPLIII} gives $J(\sigma) - J(\sigma c)$.

\emph{Dimension.} The vectors $L(\sigma)-L(\sigma\tau)$ span a
$3$-dimensional space: modulo them, $L(\sigma)$ depends only on the
unordered pair $\{\sigma(1),\sigma(2)\}$, equivalently on $\sigma(3)$, so
the six $L(\sigma)$ collapse to three classes $\ell_{1},\ell_{2},\ell_{3}$
indexed by the last input. Symmetrically the $R(\sigma)-R(\sigma\tau)$
span another $3$-dimensional space, leaving classes
$\eta_{1},\eta_{2},\eta_{3}$. Modulo these six vectors, \eqref{eq:DAPLIII} becomes
$\ell_{a}+\eta_{a} = \ell_{b}+\eta_{b}$,
which gives two independent conditions.
Hence $\dim R_{\AntiPLdual} = 3+3+2 = 8$, and
$\dim\AntiPLdual(3) = 12-8 = 4$.

\emph{Orthogonality.} Recall that the pairing we use is
diagonal, with
 \[
\langle L(\alpha), L(\beta)\rangle
= \sgn(\alpha)\delta_{\alpha\beta},\quad
\langle R(\alpha), R(\beta)\rangle = -\sgn(\alpha)\delta_{\alpha\beta} .
 \]

Three observations
suffice. First, $J$-type vectors are orthogonal to each other for this pairing:
\[
  \langle J(\alpha), J(\beta)\rangle
  = \langle L(\alpha), L(\beta)\rangle + \langle R(\alpha), R(\beta)\rangle
  = \bigl(\sgn(\alpha) - \sgn(\alpha)\bigr)\delta_{\alpha\beta} = 0 .
\]
Since \eqref{eq:AP1} and the generators of \eqref{eq:DAPLIII} are both differences of
$J(\sigma)$'s, this proves vanishing of all the corresponding pairings at once.

Second, against \eqref{eq:DAPLI} only the $L$-part of
\eqref{eq:AP1} contributes. For any $\sigma,\pi$, we have
\[
\bigl\langle L(\sigma)-L(\sigma\tau),\;
L(\pi)-L(\pi\tau)\bigr\rangle=0.
\]
Indeed, the pairing is zero unless
$\pi\in\{\sigma,\sigma\tau\}$, in which case it equals either
$\operatorname{sgn}(\sigma)+\operatorname{sgn}(\sigma\tau)$
or its negative; both vanish since
$\operatorname{sgn}(\sigma\tau)=-\operatorname{sgn}(\sigma)$.
The same argument applied to the $R$-parts handles~(14).

Third, \eqref{eq:AP2} involves only $L$-terms, so it is automatically orthogonal
to the pure-$R$ vectors of \eqref{eq:DAPLII}; against the pure-$L$ vectors of \eqref{eq:DAPLI}
one gets $\sgn(\pi)^{2} - \sgn(\pi\tau)^{2} = 0$, and against \eqref{eq:DAPLIII}
only the $L$-part contributes, giving
$\sgn(\sigma)^{2} - \sgn(\sigma c)^{2} = 0$.

Thus $R_{\AntiPLdual} \subseteq R_{\AntiPL}^{\perp}$. Since
$\dim R_{\AntiPL} + \dim R_{\AntiPLdual} = 4 + 8 = 12 = \dim\Free(E)(3)$
and the pairing is non-degenerate, the inclusion is an equality.
\end{proof}

This leads to the following definition.

\begin{definition}
    \label{def:dual-anti-prelie}
    A \emph{dual anti-pre-Lie algebra} is a vector space $D$ together with a
    bilinear operation $\star\colon D\otimes D\to D$ satisfying \eqref{eq:DAPLI}--\eqref{eq:DAPLIII} for all $a_{1}, a_{2}, a_{3}\in D$.
\end{definition}

\subsection{A Gr\"obner basis argument for dimensions of the dual anti-pre-Lie operad}\label{sec:Grobner}

Let us explain how one can study the operad $\AntiPLdual$ using Gr\"obner bases for operads \cite{DotsenkoKhoroshkin}.

\begin{theorem}
The dimension of the $n$-th component of the operad $\AntiPLdual$ is equal to $n+1$ for $n\ge 3$.
\end{theorem}

\begin{proof}

We first convert the presentation of this operad by relations \eqref{eq:DAPLI}, \eqref{eq:DAPLII}, \eqref{eq:DAPLIII} of Definition~\ref{def:dual-anti-prelie} into a shuffle operad presentation. For that, we denote
 \[
\Xsmall{1}{2}:=a_1\star a_2,\quad
\Zsmall{1}{2}:=a_2\star a_1,
 \]
and write the above relations for all permutations of arguments using this notation, obtaining the generators
\begin{gather*}
\LZZ{1}{2}{3} - \LZX{1}{2}{3},\quad
\LZZ{1}{3}{2} - \LZX{1}{3}{2},\quad
\RXZ{1}{2}{3} - \RXX{1}{2}{3},\\
\RZZ{1}{2}{3} - \LXZ{1}{3}{2},\quad
\LXX{1}{2}{3} - \LXX{1}{3}{2},\quad
\RZX{1}{2}{3} - \LXZ{1}{2}{3},\\
\LZZ{1}{2}{3} + \RZZ{1}{2}{3} - \LZZ{1}{3}{2} - \RZX{1}{2}{3},\quad
\LZZ{1}{2}{3} + \RZZ{1}{2}{3} - \RXZ{1}{2}{3} - \LXX{1}{3}{2}
\end{gather*}
of the ideal of relations of our operad
(note that permutations of relations \eqref{eq:DAPLI} and \eqref{eq:DAPLII} each give three different relations, while permutations of relation \eqref{eq:DAPLIII} give, modulo already written relations, just two different relations).

It turns out that for the reverse path-lexicographic ordering corresponding to the order of the corolla generators $\Zsmall{1}{2}>\Xsmall{1}{2}$ \cite{BremnerDotsenko}, the corresponding operad has a finite Gr\"obner basis
\begin{gather*}
\RZX{1}{2}{3} - \LXZ{1}{2}{3},\quad
\RXX{1}{2}{3} + \LXX{1}{2}{3} - \LXZ{1}{2}{3} - \LZZ{1}{3}{2},\\
\RXZ{1}{2}{3} + \LXX{1}{2}{3} - \LXZ{1}{2}{3} - \LZZ{1}{3}{2},\quad
\RZZ{1}{2}{3} - \LXZ{1}{2}{3} - \LZZ{1}{3}{2} + \LZZ{1}{2}{3},\\
\LXX{1}{3}{2} - \LXX{1}{2}{3},\quad
\LXZ{1}{3}{2} - \LXZ{1}{2}{3} - \LZZ{1}{3}{2} + \LZZ{1}{2}{3},\\
\LZX{1}{3}{2} - \LZZ{1}{3}{2},\quad
\LZX{1}{2}{3} - \LZZ{1}{2}{3},\\
\LZZZ{1}{4}{2}{3} - \LZZZ{1}{2}{4}{3},\quad
\LZZZ{1}{3}{2}{4} - \LZZZ{1}{2}{3}{4},\quad
\LZZZ{1}{4}{3}{2} - \LZZZ{1}{3}{4}{2},\\
\LXZZ{1}{3}{2}{4} - \LXZZ{1}{2}{3}{4} + \LZZZ{1}{3}{4}{2} - \LZZZ{1}{2}{4}{3},\quad
\LXZZ{1}{4}{2}{3} - \LXZZ{1}{2}{3}{4} + \LZZZ{1}{3}{4}{2} - \LZZZ{1}{2}{4}{3},\\
\LXXZ{1}{2}{3}{4} - \LZZZ{1}{3}{4}{2},
\end{gather*}
where the monomials are written in the decreasing order. This can be checked by a lengthy direct computation by hand or via the computer program \cite{DH}.

The leading terms of the first four elements of this Gr\"obner basis contain all the right combs of arity three, and hence all normal forms with respect to this Gröbner basis are left combs. Moreover, the leading terms of the seventh and the eighth elements imply that on the path from the root to the leftmost leaf of a left comb normal form all corollas $\Zsmall{}{}$ appear before all corollas $\Xsmall{}{}$. The leading terms of the fifth, the sixth, and the seventh elements guarantee that the only inversions in the planar permutation of leaves of our left comb can happen for leaves that are adjacent to corollas $\Xsmall{}{}$, and moreover, the leading terms of all but the last elements of arity four in our Gröbner basis ensure that such an inversion may only happen if the corresponding corollas $\Xsmall{}{}$ are at the root of the normal form. Finally, the last element of our Gröbner basis eliminates all among the remaining trees containing at least two copies of the corolla $\Zsmall{}{}$ and at least one copy of the corolla $\Xsmall{}{}$.

Overall, we obtain the following left comb normal forms: one with all corollas $\Zsmall{}{}$ and the unique increasing labelling of leaves, one with all corollas $\Xsmall{}{}$ except for one corolla $\Zsmall{}{}$ at the root, and the unique increasing labelling of leaves, and $n-1$ different ones with all corollas $\Xsmall{}{}$, and a labelling of leaves that increases everywhere except, possibly, for the rightmost child of the root, that is, $1,2,\ldots,i-1,i+1,\ldots,n,i$ with $2\le i\le n$. Thus, $\dim\AntiPLdual(n)=1+1+(n-1)=n+1$, as required.
\end{proof}

\subsection{Non-Koszulness of the anti-pre-Lie operad}
\label{sec:non-koszul}

The main purpose of this section is to record the following result.

\begin{theorem}
\label{thm:not-koszul}
The operad $\AntiPL$ is not Koszul.
\end{theorem}

\begin{proof}
Suppose that $\AntiPL$ is Koszul. Then Equation \eqref{eq:GKtest} implies that
 \[
f_{\AntiPL^{!}}\bigl(-f_{\AntiPL}(-t)\bigr) = t,
 \]
so that the series $f_{\AntiPL^!}(t)$ and $-f_{\AntiPL}(-t)$ are compositional inverses of one another. Computing the compositional inverse of
 \[
f_{\AntiPL^!}(t)=t+t^2+\frac23t^3+\frac5{24}t^4+\frac{1}{20}t^5+\frac7{720}t^6+\frac{1}{630}t^7+\frac{1}{4480}t^8+O(t^9),
 \]
we obtain
 \[
 t - t^2 + \frac43t^3 - \frac{15}8t^4 + \frac{38}{15}t^5 - \frac{427}{144}t^6 + \frac{10891}{5040}t^7 + \frac{13667}{5760}t^8 + O(t^9),
 \]
thus predicting
 \[
f_{\AntiPL}(t)=t + t^2 + \frac43t^3 + \frac{15}8t^4 + \frac{38}{15}t^5 + \frac{427}{144}t^6 + \frac{10891}{5040}t^7 - \frac{13667}{5760}t^8 + O(t^9),
 \]
so there is no need to compare it with the dimensions of components of the anti-pre-Lie operad we computed: it predicts a strictly negative dimension of the eighth component, and hence cannot be correct. This contradiction concludes the proof.
\end{proof}

The following corollary shows that the finite quadratic-and-cubic Gröbner basis we found for the dual anti-pre-Lie operad is the best one can hope for.

\begin{corollary}
\label{cor:no-quadratic-gb}
Neither $\AntiPL$ nor $\AntiPLdual$ admits a quadratic Gr\"obner basis, in
any admissible ordering of shuffle tree monomials.
\end{corollary}

\begin{proof}
An operad admitting a quadratic Gr\"obner basis is Koszul
\cite{DotsenkoKhoroshkin,Hoffbeck}, and a quadratic operad is Koszul if
and only if its Koszul dual is; either statement would contradict
Theorem~\ref{thm:not-koszul}.
\end{proof}

\subsection{An $S_{n}$-equivariant basis of \texorpdfstring{$\AntiPLdual(n)$}{apln}}
\label{subsec:normal-form}

The normal forms for $\AntiPLdual(n)$ exhibited in Section \ref{sec:Grobner} are not immediately useful if one wishes to write formulas compatible with the symmetric group actions, such as formulas for deformation complexes and cohomology of anti-pre-Lie algebras similar to those in \cite{KLSX}. For that reason, we give another description of components of the dual anti-pre-Lie operad.

We first note that Relation \eqref{eq:DAPLII} states that the left multiplication operators
$L_{\star}(a)$ defined by the formula
$
L_{\star}(a)(b)=a\star b$
commute in the sense that $
L_{\star}(a)L_{\star}(b)
=
L_{\star}(b)L_{\star}(a)$.
Hence, for each $j\in\{1,\dots,n\}$, the \emph{right comb}
\[
\mathbf r_j
=
L_{\star}(x_{i_1})\cdots
L_{\star}(x_{i_{n-1}})(x_j)
=
x_{i_1}\star\bigl(
x_{i_2}\star\bigl(\cdots\star
(x_{i_{n-1}}\star x_j)\cdots\bigr)\bigr),
\]
where
$
\{i_1,\dots,i_{n-1}\}
=
\{1,\dots,n\}\setminus\{j\}$,
is independent of the ordering of $i_1,\dots,i_{n-1}$ and is therefore
well defined. There are $n$ such elements, and we already know that $\dim\AntiPLdual(n)=n+1$, so
this suggests adjoining one further element in each arity and then proving that the resulting
elements form a basis.

\begin{lemma}
    \label{lem:reductions}
    Let $(D,\star)$ be a dual anti-pre-Lie algebra.
    The anti-associator
    \[
    J(a,b,c)=(a\star b)\star c+a\star(b\star c)
    \]
    is totally symmetric in $a,b,c\in D$.
\end{lemma}

\begin{proof}
    By \eqref{eq:DAPLI} and \eqref{eq:DAPLII}, we have $J(a,b,c)\equiv J(b,a,c)$,
    so $J$ is invariant under the transposition $a\leftrightarrow b$.
    Relation \eqref{eq:DAPLIII} gives $J(a,b,c)\equiv J(b,c,a)$,
    so $J$ is also invariant under the cyclic permutation
    $(a,b,c)\mapsto(b,c,a)$. Since a transposition and a $3$-cycle generate
    $S_3$, the expression $J(a,b,c)$ is totally symmetric in $a,b,c$.
\end{proof}

Inspired by this lemma and the above construction of right comb elements, we shall prove the following result.

\begin{proposition}\label{prop:Theta}
For $n\geq3$, choose a partition $\{1,\ldots,n\}=\{p,q,r\}\sqcup \{i_1,\dots,i_{n-3}\}$, and define
    \[
\Theta_n
    :=
    L_{\star}(x_{i_1})\cdots
    L_{\star}(x_{i_{n-3}})
    \bigl(J(x_p,x_q,x_r)\bigr)
    \in\AntiPLdual(n).
    \]
The element $\Theta_n$ is independent of either the choice of the
three-element subset $\{p,q,r\}$ or the ordering of the remaining
elements $i_1,\dots,i_{n-3}$.
\end{proposition}

\begin{proof}
First, we note that since left multiplication operators commute, this element is indeed independent of the ordering of $i_1,\dots,i_{n-3}$.
To prove its independence of the choice of $\{p,q,r\}$, it is enough to prove that the identity
\begin{equation}\label{eq:theta-symmetric}
 c \star J(a,b,d) = d \star J(a,b,c)
\end{equation}
holds. Indeed, this allows us to exchange one of the elements of $\{p,q,r\}$ with one of the elements of $\{i_1,\dots,i_{n-3}\}$, and iterations of such exchanges can connect any two choices of $\{p,q,r\}$.

To prove that \eqref{eq:theta-symmetric} holds, we argue as
follows. First we remark that, since 
the left multiplication operators commute, we
have
\begin{equation}\label{eq:J-left-factor}
\begin{aligned}
J(a,b,c\star d)
 &=(a\star b)\star(c\star d)+a\star\bigl(b\star(c\star d)\bigr)\\
 &=c\star\bigl((a\star b)\star d\bigr)
   +c\star\bigl(a\star(b\star d)\bigr)\\
 &=c\star J(a,b,d).
\end{aligned}
\end{equation}

Furthermore, total symmetry of $J$ and axioms of dual anti-pre-Lie algebra immediately imply
\begin{align*}
J(a,b,c\star d)&= J( c\star d, a,b)=\big( (c\star d)\star a \big)\star b+(c\star d)\star (a\star b) \\
&= \big( (d\star c)\star a \big)\star b+(d\star c)\star (a\star b)\\
&= J(a,b,d\star c).
\end{align*}
It follows that
    \[
 c \star J(a,b,d) = d \star J(a,b,c),
  \]
as required. This completes the proof.
\end{proof}

The main result of this section is the following theorem.

\begin{theorem}
For each $n\ge 3$, the $n+1$ elements $\mathbf r_1,\ldots,\mathbf r_n,\Theta_n$ form a basis of $\AntiPLdual(n)$.
\end{theorem}

We note that these elements are not obviously related to the
normal forms obtained by the Gröbner basis argument, so we
give an independent proof.

\begin{proof}
We first prove the spanning property by induction on $n$.
For $n=3$, it follows from
$(x_i\star x_j)\star x_k=J(x_i,x_j,x_k)-x_i\star(x_j\star x_k)$.
Let $n\geq4$ and write a multilinear monomial as $m=m_1\star m_2$.
If $m_1$ has arity one, apply the induction hypothesis to $m_2$;
left multiplication by the remaining variable sends each candidate
basis element to a candidate basis element of arity $n$.
If $m_1=m'_1\star m''_1$, then
\[
m=J(m'_1,m''_1,m_2)-m'_1\star(m''_1\star m_2).
\]
The second term has a left factor of smaller arity, so an inner induction
on that arity reduces the proof to the first term.

We explain the reduction of this $J$-term. Equations
 \eqref{eq:theta-symmetric} and \eqref{eq:J-left-factor}, together with
the commutativity of  the left multiplication operators, imply
\begin{align*}
(a\star b)\star J(c,d,e)&
  =e\star J(c,d,a\star b) =e\star\bigl(a\star J(c,d,b)\bigr) =a\star\bigl(e\star J(c,d,b)\bigr)\\
  & =a\star\bigl(b\star J(c,d,e)\bigr).
\end{align*}
By symmetry of $J$, \eqref{eq:J-left-factor} can be applied to a
product in any of its three inputs. Iterating it reduces the
inputs to single variables. The displayed calculation splits every
remaining composite left multiplier into multiplications by
individual variables; the left multiplication operators
commute, so this splitting also applies when other left
multipliers occur between it and $J$. Thus, for monomials $U,V,W$
on disjoint sets of variables with total arity $n$, the element
$J(U,V,W)$ is a product of left multiplications by individual
variables applied to $J$ of three variables. It equals $\Theta_n$
by Proposition~\ref{prop:Theta}. This completes the spanning
argument.

To prove the linear independence, we shall exhibit $n+1$ linear functionals on $\Free(E)(n)$ that vanish on the arity $n$ component of the ideal generated by the identities of the dual anti-pre-Lie operad. For that, we introduce two combinatorial invariants of a monomial that will be useful for us. For a multilinear monomial $m$ of arity $n$ obtained by bracketing of arguments $x_{i_1},\ldots,x_{i_n}$ (in this order), let $\rl(m)=i_n$ denote the label of its last argument, and let $k(m)$ be the depth of that argument in $m$. Set
 \[
  w(m) = \operatorname{ar}(m) - 1 - k(m) \in \mathbb{Z}_{\geq 0},
  \quad
  \varepsilon(m) = w(m) \bmod 2 \in \{0,1\}.
 \]
For instance, a right comb $\mathbf r_j$ of arity $n$ has $k=n-1$, hence $w=0$ and
$\varepsilon=0$.

We define linear functions on $\Free(E)(n)$ by their values on the monomial basis:
\[
  \phi_{\Theta}(m) = \varepsilon(m),
  \quad
  \phi_{j}(m) = \bigl(1 - 2\varepsilon(m)\bigr)\,\delta_{j,\,\rl(m)}
  \quad (j = 1, \dots, n),
\]
where $\delta$ is the Kronecker delta; thus $\phi_{j}(m) = \pm 1$ if
$x_{j}$ is the rightmost leaf of $m$, the sign being $+$ or $-$ according
as $\varepsilon(m)$ is $0$ or $1$, and $\phi_{j}(m) = 0$ otherwise.

Let us show that these linear functions are linearly independent. If we have a dependency
 \[
 c_1\phi_1+\cdots+c_n\phi_n+c_{\Theta}\phi_{\Theta}=0,
 \]
we may evaluate it on a right comb $\mathbf{r}_i$. Since $\mathbf{r}_{i}$ has $\varepsilon = 0$ and $\rl = i$, we have
$\phi_{\Theta}(\mathbf{r}_{i}) = 0$ and $\phi_{j}(\mathbf{r}_{i}) = \delta_{ij}$, which proves $c_i=0$ for $i=1,\ldots,n$.
It remains to note that $\phi_{\Theta}$ is clearly a nonzero linear function, so $c_{\Theta}$ vanishes as well.

To complete the proof, we show that the linear functions vanish on the
arity-$n$ part of the ideal generated by
\eqref{eq:DAPLI}--\eqref{eq:DAPLIII}. This ideal is spanned by expressions
$C[\rho]$, where $\rho$ is obtained by substituting multilinear monomials
$u,v,z$ into a defining relation, and $C[-]$ is a multilinear monomial
with one distinguished input slot. The variables in these substitutions
are disjoint. If $C[-]$ has arity $q$, counting the distinguished slot,
and $\rho$ has arity $p$, then $n=q+p-1$.
We use the same formulas for $\phi_\Theta$ and $\phi_j$ on monomials
in any subset of the variables, with $\phi_j=0$ when $x_j$ is absent.

Note that the coefficients of monomials in $\rho$ are precisely the coefficients of monomials in the corresponding relation of $\AntiPLdual$, and hence they add up to $0$.

\emph{Step 1: $C[-]=[-]$.} In this case, we need to prove that $\phi_{\Theta}(\rho) = \phi_{j}(\rho) = 0$. For \eqref{eq:DAPLI}, the two monomials $(uv)z$ and $(vu)z$ have the same rightmost leaf $\rl(z)$, and the same $w$, namely $\operatorname{ar}(u)+\operatorname{ar}(v)-1+w(z)$; hence both of these linear functions
take equal values on them and kill the difference. For \eqref{eq:DAPLII}, the monomials $u(vz)$ and $v(uz)$ both have rightmost leaf $\rl(z)$ and
$w = (\operatorname{ar}(u)-1) + (\operatorname{ar}(v)-1) + w(z)$, which is symmetric in $u,v$; again the difference is killed.

For \eqref{eq:DAPLIII}, written $(uv)z + u(vz) - (vz)u - v(zu)$, the first two monomials both have rightmost leaf $\rl(z)$, with
\[
  w\bigl((uv)z\bigr) = \operatorname{ar}(u)+\operatorname{ar}(v)-1+w(z),
  \quad
  w\bigl(u(vz)\bigr) = \operatorname{ar}(u)+\operatorname{ar}(v)-2+w(z),
\]
which differ by $1$; so their $\varepsilon$-values are $0$ and $1$ in some
order and sum to $1$. The last two monomials both have rightmost leaf
$\rl(u)$ and, by the same computation, $\varepsilon$-values summing to $1$.
Therefore
$\phi_{\Theta}(\rho) = 1 - 1 = 0$, while
\[
  \phi_{j}(\rho)
  = \bigl(2 - 2\cdot 1\bigr)\delta_{j,\rl(z)}
    - \bigl(2 - 2\cdot 1\bigr)\delta_{j,\rl(u)}
  = 0 .
\]

\emph{Step 2: an arbitrary context $C[-]$.}
If the distinguished slot is not the rightmost leaf of $C[-]$, every
monomial in $C[\rho]$ has the same rightmost leaf and the same depth of
that leaf. Its arity is also fixed. The functionals are therefore
constant on these monomials, and the sum of the coefficients of $\rho$
is zero, so all the functionals vanish on $C[\rho]$.

Suppose instead that the distinguished slot is the rightmost leaf.
For every monomial $v$ occurring in $\rho$, the rightmost leaf of $C[v]$
is the rightmost leaf of $v$, and
\[
k(C[v])=k(C)+k(v),\quad
w(C[v])=w(C)+w(v),
\]
where the hole is counted as a leaf in $k(C)$ and $w(C)$.
Set $c=w(C)$ and $\bar c=c\bmod2\in\{0,1\}$. As an equality of integers,
\[
\varepsilon(C[v])=\bar c+(1-2\bar c)\varepsilon(v).
\]
Consequently,
\[
\phi_\Theta(C[v])=\bar c+(1-2\bar c)\phi_\Theta(v),
\quad
\phi_j(C[v])=(-1)^c\phi_j(v).
\]
Writing $\rho=\sum_v a_vv$, we have $\sum_v a_v=0$, so
\begin{align*}
\phi_\Theta(C[\rho])
 &=\bar c\sum_v a_v+(1-2\bar c)\phi_\Theta(\rho)=0,\\
\phi_j(C[\rho])&=(-1)^c\phi_j(\rho)=0
\end{align*}
by Step~1. The functionals therefore descend to $\AntiPLdual(n)$,
which proves linear independence and completes the proof.
\end{proof}

\begin{corollary}
\label{cor:module}
For $n \geq 3$ there is an isomorphism of $S_{n}$-modules
\[
  \AntiPLdual(n) \;\cong\; \bk^{n} \oplus \bk,
\]
the first summand being the permutation module spanned by the
right combs $\mathbf{r}_{j}$,  and the second
summand being the trivial module spanned by $\Theta_n$.
\end{corollary}

\begin{proof}
Indeed, we have
$\sigma \cdot \mathbf{r}_{j} = \mathbf{r}_{\sigma(j)}$, and the element $\Theta_n$ is $S_{n}$-invariant.
\end{proof}

The equivariant basis also determines the Lie elements in the dual operad.

\begin{proposition}\label{prop:dual-metabelian}
The suboperad of $\AntiPLdual$ generated by
$[a,b]=a\star b-b\star a$ is isomorphic to the operad of metabelian
Lie algebras. In arity $n\geq 2$, its underlying $S_n$-module is
\[
  \left\{\sum_{j=1}^{n}c_j\mathbf r_j\;\middle|\;
                  \sum_{j=1}^{n}c_j=0\right\}.
\]
In particular, this component has dimension $n-1$. Here, in arity two, $\mathbf r_1=x_2\star x_1$ and
$\mathbf r_2=x_1\star x_2$.
\end{proposition}

\begin{proof}
Relation \eqref{eq:DAPLI} implies $[a,b]\star c=0$ for all $a,b,c$.
Consequently, the commutator satisfies the metabelian identity
$[[a,b],[c,d]]=0$. It also satisfies the Jacobi identity: indeed,
Relations \eqref{eq:DAPLI} and \eqref{eq:DAPLII} imply both
anti-pre-Lie identities, so the generating operation of
$\AntiPLdual$ is Lie-admissible. Thus, writing
\[
  \mathcal M=\Lie/\bigl([[x_1,x_2],[x_3,x_4]]\bigr),
\]
we obtain a surjective morphism from $\mathcal M$ onto the suboperad
under consideration.

For $n\ge2$, put
\[
c_{ij}
=\operatorname{ad}_{x_{k_1}}\cdots
\operatorname{ad}_{x_{k_{n-2}}}[x_i,x_j],
\qquad
\{k_1,\ldots,k_{n-2}\}
=\{1,\ldots,n\}\setminus\{i,j\},
\]
where $i\ne j$ and $\operatorname{ad}_x(y)=[x,y]$.
The order of the $k_\ell$ is immaterial, since the adjoint
operators commute on the derived algebra of a metabelian
Lie algebra.
By the standard basis theorem for free metabelian Lie algebras
\cite[\S 4.7.1]{MR4368864}, the elements
$c_{12},\ldots,c_{1n}$ form a basis of $\mathcal M(n)$.
Thus $\dim\mathcal M(n)=n-1$.

In $\AntiPLdual$, the identity $[a,b]\star c=0$ gives
\[
 [x,[a,b]]=x\star[a,b].
\]
Iterating this identity and using the commutativity of left
multiplications, we obtain
\[
 c_{ij}\longmapsto\mathbf r_j-\mathbf r_i.
\]
The elements $\mathbf r_2-\mathbf r_1,\ldots,
\mathbf r_n-\mathbf r_1$ are linearly independent by
Corollary~\ref{cor:module} for $n\geq 3$, and directly for $n=2$.
The above surjection is therefore an isomorphism in every arity
$n\geq 2$. It is also an isomorphism in arity one, where both spaces
are spanned by the operadic identity. The formula for the
$S_n$-module follows from $\sigma\mathbf r_j=\mathbf r_{\sigma(j)}$.
\end{proof}

\begin{remark}
	This result is parallel to the result of Mashurov and Sartayev
	\cite{MR4688832} on Lie elements in Perm-algebras:
	the suboperad of $\mathrm{Perm}$ generated by the commutator
	is isomorphic to the operad of metabelian Lie algebras.
	Recall that $\mathrm{Perm}$ is the Koszul dual of the pre-Lie operad.
\end{remark}

\subsection{Representations of dual anti-pre-Lie algebras}

Recall that for an algebraic structure defined by multilinear identities, there is a general approach going back to Eilenberg \cite{MR26647} to defining a \emph{representation} of an algebra $A$; this is equivalent to Beck modules \cite{MR1987896}. For a dual anti-pre-Lie algebra $D$, where an algebra only has one structure operation, this notion of a module becomes a vector space $V$ together with linear maps $l,r\colon D\to\mathrm{End}(V)$ which define on $D\oplus V$ a structure of a dual anti-pre-Lie algebra by the square-zero-extension formulas
\begin{equation}
\label{eq:semidirect-product}
(a+u)\bullet(b+v)=a\star b+l(a)v+r(b)u, \quad \forall\,a,b\in D,\ u,v\in V.
\end{equation}

Let us record what this definition means concretely.

\begin{proposition}
  \label{prop:semidirect}
Let $(D,\star)$ be a dual anti-pre-Lie algebra, let $V$ be a vector space and let $l,r\colon D\to\mathrm{End}(V)$ be linear maps. Formula \eqref{eq:semidirect-product} makes $D\oplus V$ a dual anti-pre-Lie algebra if and only if
the following
  identities hold for all $a,b\in D$ and $v\in V$:
  \begin{align}
    r(a)l(b)v &= r(a)r(b)v, \label{eq:rep1}\\
    l(a\star b)v &= l(b\star a)v, \label{eq:rep2}\\
    l(a)l(b)v &= l(b)l(a)v, \label{eq:rep3}\\
    l(a)r(b)v &= r(a\star b)v, \label{eq:rep4}\\
    l(a\star b)v+l(a)l(b)v
    &= r(a)l(b)v+l(b)r(a)v
    = r(a\star b)v+r(b)r(a)v. \label{eq:rep5}
  \end{align}
\end{proposition}

\begin{proof}
  Expand each of the relations \eqref{eq:DAPLI}, \eqref{eq:DAPLII}, \eqref{eq:DAPLIII} of
  Definition~\ref{def:dual-anti-prelie} for $\bullet$ on $D\oplus V$ and
  collect the components lying in $V$. Since $\bullet$ restricts to
  $\star$ on $D$, the components lying in $D$ hold automatically, and the
  $V$-components are, term by term, the identities
  \eqref{eq:rep1}--\eqref{eq:rep5}.
\end{proof}

For the dual anti-pre-Lie algebra structures produced by this proposition, we write $D\ltimes_{l,r}V=(D\oplus V,\bullet)$ and call this algebra the \emph{semidirect product} of $(D,\star)$ by $(l,r,V)$.

\begin{example}
  \label{ex:adjoint-rep}
  Let $(D,\star)$ be a dual anti-pre-Lie algebra. Then
  $(L_{\star},R_{\star},D)$ is a representation of $(D,\star)$, called
  the \emph{adjoint representation}.
\end{example}

\begin{proposition}
  \label{prop:dual-rep}
  Let $(D,\star)$ be a dual anti-pre-Lie algebra. If $(l,r,V)$ is a
  representation of $(D,\star)$, then so is
  $(l^{*},l^{*}-r^{*},V^{*})$, where $l^{*}\colon D\to\mathrm{End}(V^{*})$
  is given by
  \begin{equation}
    \label{eq:dual-map}
    \langle l^{*}(a)u^{*},v\rangle=-\langle u^{*},l(a)v\rangle,
    \quad \forall\,a\in D,\ u^{*}\in V^{*},\ v\in V,
  \end{equation}
  and $r^{*}$ is defined in the same way. In particular,
  $(L^{*}_{\star},L^{*}_{\star}-R^{*}_{\star},D^{*})$ is a representation
  of $(D,\star)$.
\end{proposition}

\begin{proof}
  For all $a,b\in D$, $u^{*}\in V^{*}$ and $v\in V$,
  \begin{align*}
    \bigl\langle
    (l^{*}-r^{*})(a)l^{*}(b)u^{*}
    -(l^{*}-r^{*})(a)(l^{*}-r^{*})(b)u^{*},\,v
    \bigr\rangle
    &=\bigl\langle (l^{*}-r^{*})(a)r^{*}(b)u^{*},\,v\bigr\rangle\\
    &=\bigl\langle u^{*},\,r(b)l(a)v-r(b)r(a)v\bigr\rangle\\
    &\overset{\eqref{eq:rep1}}{=}0,
  \end{align*}
  that is,
  \[
  (l^{*}-r^{*})(a)l^{*}(b)u^{*}
  =(l^{*}-r^{*})(a)(l^{*}-r^{*})(b)u^{*}.
  \]
  Hence \eqref{eq:rep1} holds for $(l^{*},l^{*}-r^{*},V^{*})$. The
  remaining identities \eqref{eq:rep2}--\eqref{eq:rep5} are verified in
  the same way, and the last statement is the case
  $(l,r,V)=(L_{\star},R_{\star},D)$ of Example~\ref{ex:adjoint-rep}.
\end{proof}

\subsection{The cyclic structure of the dual anti-pre-Lie operad and quadratic structures}\label{sec:cyclic}

It is well known \cite{MR2135545,MR1913297} that the Koszul dual of a cyclic operad $\calP$ is again cyclic. Thus, the operad
$\AntiPLdual$ is cyclic. Explicitly, its cyclic structure translates to the following definition.

\begin{definition}
A symmetric bilinear form $S$ on a dual anti-pre-Lie algebra $D$ is said to be \emph{invariant} if
\begin{equation}\label{eq:S-invariance}
    S(a\star b,c)=S(a,\,c\star b-b\star c),
    \quad \forall\,a,b,c\in D.
\end{equation}
\end{definition}

\begin{remark}
\label{rem:S-consequence}
As in the case of anti-pre-Lie algebras, the invariance property implies another relation allowing one to isolate the other argument of the product:
  \begin{equation}\label{eq:S-consequence}
    S(a\star b,c)+S(b,\,a\star c)=0,
    \quad \forall\,a,b,c\in D.
  \end{equation}
\end{remark}

Recall that a \emph{quadratic Lie algebra} is a Lie algebra
$(\mathfrak{g},[-,-])$ together with a nondegenerate symmetric bilinear
form $K$ on $\mathfrak{g}$ which is \emph{invariant} in the
sense that
\begin{equation}
  \label{eq:K-invariance}
  K([x,y],z)=K(x,[y,z]),
  \quad \forall\,x,y,z\in\mathfrak{g}.
\end{equation}
In general, such a notion of a quadratic algebra is available for any cyclic operad $\calP$: a \emph{quadratic $\calP$-algebra} is a $\calP$-algebra equipped with a nondegenerate invariant symmetric bilinear form. In particular, a \emph{quadratic
anti-pre-Lie algebra} is an anti-pre-Lie algebra equipped with a nondegenerate invariant symmetric bilinear form, and a \emph{quadratic
dual anti-pre-Lie algebra} is a dual anti-pre-Lie algebra equipped with a nondegenerate invariant symmetric bilinear form.

Let us relate the definition of a quadratic dual anti-pre-Lie algebra to the previously defined notion of the dual representation.

\begin{proposition}\label{prop:quadratic-equiv}
Let $(D,\star)$ be a finite-dimensional dual anti-pre-Lie algebra,
and let $S$ be a nondegenerate bilinear form on $D$.
Let $\phi\colon D\to D^*$ be the induced linear isomorphism given by
\begin{equation}\label{eq:phi-def}
\langle\phi(a),b\rangle=S(a,b),\quad\forall\,a,b\in D.
\end{equation}
Define $\phi^t\colon D\to D^*$ by
$\langle\phi^t(a),b\rangle=\langle\phi(b),a\rangle$.
Then $(D,\star,S)$ is a quadratic dual anti-pre-Lie algebra if and
only if $\phi^t=\phi$ and $\phi$ gives an equivalence between
$(L_\star,R_\star,D)$ and
$(L_\star^*,L_\star^*-R_\star^*,D^*)$.
\end{proposition}

\begin{proof}
The nondegeneracy of $S$ and finite dimensionality of $D$ imply that
$\phi$ is an isomorphism.
Suppose that  $(D,\star,S)$  is a quadratic dual anti-pre-Lie algebra.
Since $S$ is symmetric, we have
$\phi^t=\phi$.
Moreover, by \eqref{eq:S-invariance} and \eqref{eq:S-consequence},
we have
\begin{align*}
\langle\phi L_\star(a)b,c\rangle
 &=S(a\star b,c)=-S(b,a\star c)
 =\langle L_\star^*(a)\phi(b),c\rangle,\\
\langle\phi R_\star(a)b,c\rangle
 &=S(b\star a,c)=S(b,c\star a-a\star c) =\langle(L_\star^*-R_\star^*)(a)\phi(b),c\rangle,\;\forall a,b,c\in D.
\end{align*}
Thus $\phi$ gives the stated equivalence of representations.
The converse is obtained by reversing
these computations.
\end{proof}

\begin{proposition}
    \label{prop:double}
    Let $(D,\star)$ be a dual anti-pre-Lie algebra and let
    $D\ltimes_{L^{*}_{\star},\,L^{*}_{\star}-R^{*}_{\star}}D^{*}$ be the
    semidirect product of $(D,\star)$ by the representation
    $(L^{*}_{\star},L^{*}_{\star}-R^{*}_{\star},D^{*})$ of
    Proposition~\ref{prop:dual-rep}. Then
    $\bigl(D\ltimes_{L^{*}_{\star},\,L^{*}_{\star}-R^{*}_{\star}}D^{*},
    S_{d}\bigr)$ is a quadratic dual anti-pre-Lie algebra, where
    \begin{equation}
        \label{eq:Sd}
        S_{d}(a+\alpha,\,b+\beta)
        =\langle a,\beta\rangle+\langle\alpha,b\rangle,
        \quad \forall\,a,b\in D,\ \alpha,\beta\in D^{*}.
    \end{equation}
\end{proposition}

\begin{proof}
    The form $S_{d}$ is visibly symmetric and nondegenerate.
    Moreover, for all $a,b,c\in D$ and $\alpha,\beta,\gamma\in D^{*}$,
    \begin{align*}
        S_{d}\bigl((a+\alpha)\bullet(b+\beta),\,c+\gamma\bigr)
        &=S_{d}\bigl(a\star b+L^{*}_{\star}(a)\beta
        +(L^{*}_{\star}-R^{*}_{\star})(b)\alpha,\;c+\gamma\bigr)\\
        &=\langle a\star b,\gamma\rangle
        +\langle L^{*}_{\star}(a)\beta,c\rangle
        +\bigl\langle(L^{*}_{\star}-R^{*}_{\star})(b)\alpha,c\bigr\rangle\\
        &=\langle a\star b,\gamma\rangle
        -\langle\beta,a\star c\rangle
        +\langle\alpha,\,c\star b-b\star c\rangle\\
        &=-\langle a,R^{*}_{\star}(b)\gamma\rangle
        +\langle R^{*}_{\star}(c)\beta,a\rangle
        +\langle\alpha,\,c\star b-b\star c\rangle\\
        &=S_{d}\bigl(a+\alpha,\;
        (c+\gamma)\bullet(b+\beta)-(b+\beta)\bullet(c+\gamma)\bigr).
    \end{align*}
    Hence $S_{d}$ is invariant, and the claim follows.
\end{proof}

We conclude with the following result.

\begin{proposition}
    \label{prop:tensor-quadratic}
    Let $(A,\circ,B)$ be a quadratic anti-pre-Lie algebra and let
    $(D,\star,S)$ be a quadratic dual anti-pre-Lie algebra. Then
    $(A\otimes D,[-,-],K)$ is a quadratic Lie algebra, where
    \begin{align}
        [x\otimes a,\,y\otimes b]
        &=x\circ y\otimes a\star b-y\circ x\otimes b\star a,
        \label{eq:tensor-bracket}\\
        K(x\otimes a,\,y\otimes b)
        &=B(x,y)\,S(a,b),
        \quad \forall\,x,y\in A,\ a,b\in D.
        \label{eq:tensor-form}
    \end{align}
\end{proposition}

\begin{proof}
    By \cite[Proposition~7.6.5]{LodayVallette}, $(A\otimes D,[-,-])$ is a
    Lie algebra, and $K$ is clearly a nondegenerate symmetric
    bilinear form on $A\otimes D$. For all $x,y,z\in A$ and $a,b,c\in D$,
    \begin{align*}
        &K\bigl([x\otimes a,y\otimes b],\,z\otimes c\bigr)
        +K\bigl(y\otimes b,\,[x\otimes a,z\otimes c]\bigr)\\
        &\quad=K\bigl(x\circ y\otimes a\star b
        -y\circ x\otimes b\star a,\;z\otimes c\bigr)
        +K\bigl(y\otimes b,\;x\circ z\otimes a\star c
        -z\circ x\otimes c\star a\bigr)\\
        &\quad=B(x\circ y,z) S (a\star b,c)
        -B(y\circ x,z)S(b\star a,c)
        +B(y,x\circ z)S(b,a\star c)\\
        &\quad\phantom{={}}
        -B(y,z\circ x)S(b,c\star a)\\
        &\quad\overset{\eqref{eq:APL-inv},\eqref{eq:B-consequence}}{=}
        B(y,\,x\circ z-z\circ x) S (a\star b,c)
        +B(y,z\circ x)S(b\star a,c)
        +B(y,x\circ z)S(b,a\star c)\\
        &\quad\phantom{={}}
        -B(y,z\circ x)S(b,c\star a)\\
        &\quad\overset{\eqref{eq:S-invariance},\eqref{eq:S-consequence}}{=}0 .
    \end{align*}
Hence the bilinear form $K$ is invariant, and $(A\otimes D,[-,-],K)$
    is a quadratic Lie algebra.
\end{proof}

\begin{remark}
Operadically, this result means that the usual morphism
 \[
\Lie\to\AntiPL \mathbin{\underset{\mathrm H}{\otimes}}\AntiPLdual
 \]
from the Lie operad to the Hadamard product of $\AntiPL$ and $\AntiPLdual$ is a morphism of cyclic operads. In fact, the same is true for any cyclic operad $\calP$ generated by binary operations satisfying only quadratic relations: the canonical morphism
 \[
\Lie\to \calP\mathbin{\underset{\mathrm H}{\otimes}}\calP^!
 \]
is a morphism of cyclic operads. This can be justified by a computation similar to the one above, or alternatively one can note that this morphism arises from the morphism from the operad $\Lie$ to the white Manin product of $\calP$ and $\calP^!$, which comes from the adjunction \cite{GinzburgKapranov,MR2427978}
 \[
\Hom(\calP\medbullet\calQ,\calR)\cong\Hom(\calP,\calQ^!\medcirc\calR)
 \]
by writing
 \[
\Hom(\Lie,\calP^!\medcirc\calP)\cong
\Hom(\Lie\medbullet\calP,\calP)\cong\Hom(\calP\medbullet\Lie,\calP)\cong\Hom(\calP,\Com\medcirc\calP)\cong\Hom(\calP,\calP).
 \]
A direct inspection shows that the Manin product adjunction is compatible with cyclic actions, furnishing a conceptual explanation of the above result.
\end{remark}

\section{Open questions}
\label{sec:open}

We conclude this paper by stating several conjectures and outlining directions for future work.

\subsection{Identities for binary operations in the anti-pre-Lie operad}

\subsubsection{The skew-symmetric operation}
We saw in Theorem \ref{thm:lie-kernel} that the lowest degree in which the commutator of a general anti-pre-Lie algebra satisfies an identity not implied by the Lie axioms is five, and in that degree all identities follow from the standard identity of degree five. We conjecture that there are no further identities.

\begin{conjecture}\label{conj:LieAPL}
The suboperad of $\AntiPL$ generated by the Lie bracket $a_1\circ a_2-a_2\circ a_1$ is isomorphic to the quotient of the Lie operad by the ideal generated by the standard identity of degree five.
\end{conjecture}

It is also of independent interest to describe the kernel of the map $\kappa$ as a collection of representations of symmetric groups. For the kernel of the map $\Lie\to\Nov$, which is an upper bound for $\ker\kappa$, this is the result of Molev \cite[Theorem~1]{Molev} mentioned in the proof of Theorem \ref{thm:lie-kernel}. Using the result of Molev, we checked that the kernel of the map $\Lie\to\Nov$ is generated by the standard identity in arities $n\le 9$; consequently, Conjecture \ref{conj:LieAPL} holds in arities $n\le 9$.

Another related question is to characterize Lie algebras $L$ that can be embedded in commutator Lie algebras of anti-pre-Lie algebras. In the pre-Lie case, that is the case for every Lie algebra \cite{Segal}; for  anti-pre-Lie algebras, the standard identity of degree five clearly must hold for $L$, but it is not clear if there are other constraints.

\subsubsection{The symmetric operation}

It is also interesting to consider the symmetric product $a\circ b+b\circ a$ that one can define on anti-pre-Lie algebras.  Does it satisfy any identities? A direct computer calculation shows that this product satisfies no identities beyond commutativity in degrees at most $5$.

\begin{conjecture}
The suboperad of the operad $\AntiPL$ generated by the symmetric product $a\circ b+b\circ a$ is isomorphic to the commutative magmatic operad; in other words, the symmetric product on a general anti-pre-Lie algebra satisfies no identities beyond commutativity.
\end{conjecture}

Note that in the case of Novikov algebras, the symmetric product $a\star b+b\star a$ satisfies a polynomial identity of degree $4$, known as the (right) Tortken identity \cite{DzhumadildaevNovikovJordan}, while  for pre-Lie algebras, the symmetric product is known to satisfy no identities beyond commutativity \cite{MR2763748}.

\subsection{Identities for the symmetric operation in the dual anti-pre-Lie operad}

Computations in arities $n\le 10$ suggest the following conjecture.

\begin{conjecture}
The suboperad of $\AntiPLdual$ generated by the symmetric product $a\star b+b\star a$ is isomorphic to the quotient of the commutative magmatic operad by the ideal generated by the two cubic relations
 \[
2((a_1a_2)a_3)a_4-2((a_1a_3)a_2)a_4+((a_1a_4)a_2)a_3-((a_1a_4)a_3)a_2,
 \]
 \[
(a_1a_2)(a_3a_4)-(a_1a_3)(a_2a_4).
 \]
\end{conjecture}

\subsection{Homotopy and cohomology}

\subsubsection{Minimal models}\mbox{}\par

It would be interesting to construct an explicit minimal model $\AntiPL_{\infty}$ of $\AntiPL$ and develop the resulting theory of homotopy anti-pre-Lie algebras, though this problem might be extremely hard. The same question for the operad $\AntiPLdual$ may be more tractable, since that latter operad admits a finite Gröbner basis, and methods of \cite{DKQuillen} may be applied.

\subsubsection{Cohomology and deformations of anti-pre-Lie algebras} Similarly to how it is done for Novikov algebras in \cite{KLSX}, one may use the operadic Chevalley--Eilenberg cohomology to define a cohomology theory of anti-pre-Lie algebras; since an explicit description of the Koszul dual operad together with the symmetric group actions
is furnished by Corollary \ref{cor:module}, this seems a realistic task.

\subsection{Free algebras}

It would be interesting to find a combinatorial description of free anti-pre-Lie algebras, parallel in spirit to the rooted trees model for the pre-Lie operad \cite{ChapotonLivernet}.  However, Proposition~\ref{prop:dimensions} already rules out a model coming from an ordinary set-valued species whose linearization is $\AntiPL(n)$ and on which $S_n$ acts by relabelling basis objects.  Indeed,
\[
\AntiPL(3)\cong 2S^{3}\oplus 3S^{2,1}.
\]
If $c=(123)\in S_3$, then the trivial module $S^3$ has trace $1$ on $c$, whereas the standard module $S^{2,1}$ has trace $-1$.  Hence
\[
\operatorname{tr}\bigl(c\mid \AntiPL(3)\bigr)=2\cdot 1+3\cdot(-1)=-1.
\]
By contrast, the trace of any permutation of a basis is the number of basis elements it fixes, and is therefore a nonnegative integer.  Thus $\AntiPL(3)$ is not a permutation representation, so no such relabelling basis can exist.

The open problem should therefore be understood more broadly: one may look for a combinatorial or representation-theoretic model allowing signs or more general linear $S_n$-actions, or for another explicit description of free anti-pre-Lie algebras.  The dimensions of the labelled multilinear components are
\[
1, 2, 8, 45, 304, 2283, 18940, \dots .
\]
Already describing the free anti-pre-Lie algebra on one generator is an open problem.  Proposition~\ref{prop:dimensions} implies that the dimensions of its homogeneous components are
\[
1,1,2,4,9,18,39,\ldots .
\]
In characteristic zero these numbers are the dimensions of the coinvariants $\AntiPL(n)_{S_n}$. Equivalently, they are the multiplicities of the trivial representation in $\AntiPL(n)$. Without a permutation-basis model, they should not be interpreted automatically as numbers of unlabelled combinatorial objects.  We did not find either sequence in the OEIS \cite{oeis} (accessed October 1, 2026).  By contrast, Corollary~\ref{cor:module} immediately leads to a simple description of free dual anti-pre-Lie algebras.

\noindent{\bf Acknowledgements.} This work is supported by
Institut Universitaire de France and NSFC (12271265, W2412041,
12671037). It is also supported by the China Postdoctoral Science
Foundation (2024T005TJ, 2024M761507), the Postdoctoral Fellowship
Program of CPSF (GZC20240755), the Fundamental Research Funds for
the Central Universities and Nankai Zhide Foundation.

 \smallskip
 \noindent
 {\bf Declaration of interests. } The authors have no conflicts of interest to
 disclose.

 \smallskip

 \noindent
 {\bf Data availability. } No new data were created or analyzed in this study.

\appendix

\section{Proof of the standard identity of degree five}\label{app:standard}

In this appendix, we record a detailed proof of the fact that the standard identity of degree five $\mathrm{st}_5(x_{1},\dots,x_{5})$ is satisfied in every commutator Lie algebra of an anti-pre-Lie algebra.
\begin{lemma}\label{lem:ad-comm}
  Let $(A,\circ)$ be an anti-pre-Lie algebra. Then the following identity holds:
  \begin{align}
    &\ad([x,y])
    = L_{\circ}(y)R_{\circ}(x)-L_{\circ}(x)R_{\circ}(y),\label{eq:ad-LR}
  \end{align}
where
\begin{equation*}
L_{\circ}(x)(y) = x \circ y,\;
R_{\circ}(x)(y) = y \circ x,\;
\ad(x)(y) = [x,y]=x \circ y-y\circ x,\;\forall x,y\in A.
\end{equation*}
\end{lemma}
\begin{proof}
For all $x,y,z\in A$, we have
\begin{align*}
&[[x,y],z]+x\circ(z\circ y)-y\circ(z\circ x)\\
&=[x,y]\circ z-z\circ [x,y] +x\circ (z\circ y)-y\circ(z\circ x)\\
&\overset{\eqref{eq:AP1}}{=}y\circ(x\circ z)-x\circ(y\circ z)-z\circ [x,y]+x\circ(z\circ y)-y\circ (z\circ x)\\
&=x\circ [z,y]+y\circ [x,z]+z\circ [y,x]\\
&\overset{\eqref{eq:altI}}{=}0,
\end{align*}
which equivalently gives \eqref{eq:ad-LR}.
\end{proof}

For a multilinear expression in $x_{1},\dots,x_{5}$ we write
\[
\langle m(x_{1},\cdots,x_{5}) \rangle
:= \sum_{\sigma \in S_{4}} \sgn(\sigma)\,
m\bigl(x_{\sigma(1)},\dots,x_{\sigma(4)},x_{5}\bigr)
\]
for the alternation over the first four arguments, $x_{5}$ being fixed;
$\langle - \rangle$ is extended linearly and is always understood in
$\AntiPL(5)$. We use constantly that
$\langle \tau\cdot m\rangle = \sgn(\tau)\langle m\rangle$ for
$\tau \in S_{4}$; in particular, if $F$ is linear in one slot and does
not involve $x_{1},x_{2}$, then
\begin{equation}
  \label{eq:halve}
 2\, \bigl\langle F(x_{1}\circ x_{2})\bigr\rangle
  =  \bigl\langle F([x_{1},x_{2}])\bigr\rangle ,
\end{equation}
and similarly for $x_{3},x_{4}$.

Using the left
and right multiplication operators, we can reformulate the
standard identity in a very particular way.

\begin{lemma}\label{lem:reduction}
  In any anti-pre-Lie algebra,
  \begin{equation}
    \label{eq:LRLR}
    \mathrm{st}_5(x_{1},\dots,x_{5})
    =
    \Bigl\langle\, x_{1} \circ
    \bigl((x_{3} \circ (x_{5} \circ x_{4}))\circ x_{2}\bigr)\Bigr\rangle
    =
    \Alt_{4}\bigl(
    L_{\circ}(x_{1})R_{\circ}(x_{2})
    L_{\circ}(x_{3})R_{\circ}(x_{4})
    \bigr)(x_{5}).
  \end{equation}
\end{lemma}

\begin{proof}
  Put
  \[
  W=
  \Alt_{4}\bigl(
  \ad(x_{1})\ad(x_{2})\ad(x_{3})\ad(x_{4})
  \bigr).
  \]
  Replacing $\sigma$ by $\sigma\cdot(1\,2)$ flips the sign and exchanges
  the first two factors, so
  \[
  2\, W
  =
   \Alt_{4}\bigl(
  [\ad(x_{1}),\ad(x_{2})]
  \ad(x_{3})\ad(x_{4})
  \bigr).
  \]
  Repeating with $(3\,4)$ and using $[\ad(u),\ad(v)]=\ad([u,v])$,
  we obtain
  \[
 4\, W
  =
  \sum_{\sigma\in S_{4}}\sgn(\sigma)\,
  \ad([x_{\sigma(1)},x_{\sigma(2)}])
  \ad([x_{\sigma(3)},x_{\sigma(4)}]).
  \]
 By Lemma~\ref{lem:ad-comm}, we have
  \[
  \ad([x_{\sigma(1)},x_{\sigma(2)}])
  =
  L_{\circ}(x_{\sigma(2)})R_{\circ}(x_{\sigma(1)})
  -
  L_{\circ}(x_{\sigma(1)})R_{\circ}(x_{\sigma(2)}).
  \]
  Noting that $\sigma\mapsto\sigma\cdot(1\,2)$ fixes the second factor
  while exchanging the two summands and flipping the sign, we have
  \[
  4\,W
  =
  -2\sum_{\sigma\in S_{4}}\sgn(\sigma)\,
  L_{\circ}(x_{\sigma(1)})R_{\circ}(x_{\sigma(2)})
  \ad([x_{\sigma(3)},x_{\sigma(4)}]).
  \]
  Doing the same for the second factor with $(3\,4)$ yields
  \[
  4\,W
  =
  4\,\Alt_{4}\bigl(
  L_{\circ}(x_{1})R_{\circ}(x_{2})
  L_{\circ}(x_{3})R_{\circ}(x_{4})
  \bigr).
  \]
  Evaluating at $x_{5}$ gives \eqref{eq:LRLR}.
\end{proof}

The whole question is now reduced to proving the vanishing of a single alternated monomial,
which is the content of Proposition~\ref{prop:alt-s4} below.

Let us fix the notation and
isolate the one identity that drives the whole argument.
We abbreviate
\[
Y=[x_{1},x_{2}],
\;
W=[x_{3},x_{4}],
\;
Q=[x_{5},Y],
\;
Z=[x_{5},W],
\;
P=x_{3}\circ(x_{5}\circ x_{4}),
\]
and we set
\[
A:=\langle Z\circ Y\rangle,
\;
B:=\langle Y\circ Z\rangle,
\;
C:=\bigl\langle x_{1}\circ(P\circ x_{2})\bigr\rangle.
\]

\begin{lemma}[$P$-to-$Z$ reduction]
    \label{lem:PZ}
    In any anti-pre-Lie algebra,
    \begin{equation}
        \label{eq:P-minus-P34}
        P-P^{(3\,4)}=Z.
    \end{equation}
    Consequently, for any $F$ linear in one slot and independent of
    $x_{3},x_{4}$,
    \begin{equation}
        \label{eq:step0-PZ}
        2\,\langle F(P)\rangle
        =\langle F(Z)\rangle.
    \end{equation}
\end{lemma}

\begin{proof}
    Applying \eqref{eq:AP1} to $(x_{3},x_{5},x_{4})$ and
    $(x_{4},x_{5},x_{3})$ and subtracting gives
    \[
    P-P^{(3\,4)}-x_{5}\circ W
    =
    [x_{5},x_{3}]\circ x_{4}
    -
    [x_{5},x_{4}]\circ x_{3}.
    \]
    By \eqref{eq:AP2} applied to $(x_{5},x_{3},x_{4})$, the right-hand side
    is $-W\circ x_{5}$. Hence
    \[
    P-P^{(3\,4)}
    =
    x_{5}\circ W-W\circ x_{5}
    =
    [x_{5},W]
    =
    Z,
    \]
    which is \eqref{eq:P-minus-P34}. Since
    $\langle F(P^{(3\,4)})\rangle=-\langle F(P)\rangle$,
    we obtain \eqref{eq:step0-PZ}.
\end{proof}

\begin{lemma}[Subtree lemma]
    \label{lem:subtree}
    Let $m$ be a multilinear monomial in $x_{1},\dots,x_{5}$.
    If $m$ has a subtree whose set of leaves is exactly three or exactly four of
    $x_{1},\dots,x_{4}$ \textup{(}and does not contain $x_{5}$\textup{)},
    then $\langle m\rangle=0$.
\end{lemma}

\begin{proof}
    The alternation over $x_1,\dots,x_4$ can be performed by first
    alternating over the leaves of the given subtree, and
    Proposition~\ref{prop:nosgn} ensures that the latter alternation vanishes.
\end{proof}

We are now ready to prove that the standard identity in the form of
Lemma \ref{lem:reduction} holds.

\begin{proposition}
    \label{prop:alt-s4}
    In $\AntiPL(5)$ one has
    \[
    \bigl\langle
    x_{1}\circ\bigl(P\circ x_{2}\bigr)
    \bigr\rangle=0.
    \]
\end{proposition}
\begin{proof}
    Applying \eqref{eq:AP1} to $(x_1,P,x_2)$ and alternating gives
    \[
    C
    =
    \bigl\langle P\circ(x_1\circ x_2)\bigr\rangle
    +
    \bigl\langle [P,x_1]\circ x_2\bigr\rangle.
    \]
    By \eqref{eq:halve} and Lemma~\ref{lem:PZ},
    \[
    4\bigl\langle P\circ(x_1\circ x_2)\bigr\rangle
    =
    2\langle P\circ Y\rangle
    =
    \langle Z\circ Y\rangle
    =
    A.
    \]
    On the other hand, \eqref{eq:AP2} applied to $(P,x_1,x_2)$ gives
    \[
    [P,x_1]\circ x_2
    +
    Y\circ P
    +
    [x_2,P]\circ x_1
    =
    0.
    \]
    By $(1\,2)$-antisymmetry, the first and third terms have equal alternations.\par
    \noindent Hence, using Lemma~\ref{lem:PZ} again,
    \[
    4\bigl\langle [P,x_1]\circ x_2\bigr\rangle
    =
    -2\langle Y\circ P\rangle
    =
    -B.
    \]
    Consequently, $4C=A-B$.

    We next compute $B$ in two ways.
    By Lemma~\ref{lem:PZ} and successive applications of
    \eqref{eq:AP1} to $(Y,x_3,x_5\circ x_4)$ and $(Y,x_5,x_4)$,
    we obtain
    \[
        B
         =2\langle Y\circ P\rangle =2\bigl\langle
        x_3\circ\bigl(Y\circ(x_5\circ x_4)\bigr)
        \bigr\rangle =2\bigl\langle x_3\circ(Q\circ x_4)\bigr\rangle.
    \]
    The terms omitted in the last two equalities contain
    $[x_3,Y]$ and $Y\circ x_4$, respectively.
    After expanding the commutators, the corresponding subtrees
    have leaves $x_1,x_2,x_3$ and $x_1,x_2,x_4$,
    so their alternations vanish by Lemma~\ref{lem:subtree}.
    The even permutation $(1\,3)(2\,4)$ sends
    $x_3\circ(Q\circ x_4)$ to $x_1\circ(Z\circ x_2)$.
    Thus Lemma~\ref{lem:PZ} gives
    \[
    B
    =
    2\bigl\langle x_3\circ(Q\circ x_4)\bigr\rangle
    =
    2\bigl\langle x_1\circ(Z\circ x_2)\bigr\rangle
    =
    4C.
    \]

    On the other hand, \eqref{eq:AP2} applied to $(x_5,Y,x_4)$ gives
    \[
    Q\circ x_4
    =
    -[Y,x_4]\circ x_5
    -
    [x_4,x_5]\circ Y.
    \]
    After left multiplication by $x_3$, the first term vanishes
    under alternation by Lemma~\ref{lem:subtree}, since
    $[Y,x_4]$ involves only $x_1,x_2,x_4$.
    Applying \eqref{eq:AP1} to $(x_3,[x_4,x_5],Y)$ then yields
    \[
        B
    =-2\bigl\langle
        x_3\circ\bigl([x_4,x_5]\circ Y\bigr)
        \bigr\rangle
    =-2\bigl\langle [[x_4,x_5],x_3]\circ Y\bigr\rangle.
    \]
    In the last equality, the omitted term
    $[x_4,x_5]\circ(x_3\circ Y)$ has zero alternation
    by Lemma~\ref{lem:subtree}, since $x_3\circ Y$
    has leaves $x_1,x_2,x_3$.

    Finally, the Jacobi identity gives
    \[
    [[x_4,x_5],x_3]-[[x_3,x_5],x_4]
    =
    [x_5,W]
    =
    Z.
    \]
    Using $(3\,4)$-antisymmetry and noting that $Y$ is fixed
    by $(3\,4)$, we obtain
    \[
    2\bigl\langle [[x_4,x_5],x_3]\circ Y\bigr\rangle
    =
    \langle Z\circ Y\rangle
    =
    A.
    \]
    Therefore $B=-A$. Combining this with
    $B=4C$ and $4C=A-B$, we conclude that
    \[
    B=4C=A-B=-2B.
    \]
    Thus $3B=0$. Since we work in characteristic zero,
    $B=0$, and hence $C=0$.
\end{proof}

\begingroup
\emergencystretch=1.5em
\raggedbottom
\renewcommand{\bibliofont}{\normalfont\normalsize}

\endgroup


\begin{thebibliography}{48}
\setlength{\itemsep}{0pt}
\setlength{\parsep}{0pt}

\bibitem{Aguiar}
M. Aguiar.
``Pre-{P}oisson algebras''.
In: \emph{Lett. Math. Phys.} 54.4 (2000), pp.~263--277.

\bibitem{AlbertPAss}
A. A. Albert.
``Power-associative rings''.
In: \emph{Trans. Amer. Math. Soc.} 64 (1948), pp.~552--593.

\bibitem{MR4368864}
Y. Bahturin.
\emph{Identical relations in {L}ie algebras}.
Vol.~68.
De Gruyter Expositions in Mathematics,
Second edition.
De Gruyter, Berlin, [2021] \copyright 2021, pp.~xxv+514.

\bibitem{Bai}
C. Bai.
``An introduction to pre-{L}ie algebras''.
In: \emph{Algebra and applications 1---non-associative algebras and categories}.
ISTE, London, 2020, pp.~245--273.

\bibitem{BG1}
C. Bai and D. Gao.
``Graded anti-pre-{L}ie algebraic structures on {W}itt and {V}irasoro algebras''.
In: \emph{J. Geom. Phys.} 214 (2025), Paper No. 105525, 15.

\bibitem{BG2}
C. Bai and D. Gao.
``Compatible root graded anti-pre-{L}ie algebraic structures on finite-dimensional complex simple {L}ie algebras''.
In: \emph{Int. J. Geom. Methods Mod. Phys.} 23.5 (2026), Paper No. 2540039, 20.

\bibitem{BaN}
A. A. Balinskii and S. P. Novikov.
``Poisson brackets of hydrodynamic type, {F}robenius algebras and {L}ie algebras''.
In: \emph{Dokl. Akad. Nauk SSSR} 283.5 (1985), pp.~1036--1039.

\bibitem{MR1987896}
J. M. Beck.
``Triples, algebras and cohomology''.
In: \emph{Repr. Theory Appl. Categ.} 2 (2003), pp.~1--59.

\bibitem{MR2763748}
N. Bergeron and J.-L. Loday.
``The symmetric operation in a free pre-{L}ie algebra is magmatic''.
In: \emph{Proc. Amer. Math. Soc.} 139.5 (2011), pp.~1585--1597.

\bibitem{BremnerDotsenko}
M. R. Bremner and V. Dotsenko.
\emph{Algebraic operads}.
An algorithmic companion.
CRC Press, Boca Raton, FL, 2016, pp.~xvii+365.

\bibitem{BMP}
M. R. Bremner, S. Madariaga, and L. A. Peresi.
``Structure theory for the group algebra of the symmetric group, with applications to polynomial identities for the octonions''.
In: \emph{Comment. Math. Univ. Carolin.} 57.4 (2016), pp.~413--452.

\bibitem{Bur}
D. Burde.
``Left-symmetric algebras, or pre-{L}ie algebras in geometry and physics''.
In: \emph{Cent. Eur. J. Math.} 4.3 (2006), pp.~323--357.

\bibitem{MR2135545}
F. Chapoton.
``On some anticyclic operads''.
In: \emph{Algebr. Geom. Topol.} 5 (2005), pp.~53--69.

\bibitem{ChapotonLivernet}
F. Chapoton and M. Livernet.
``Pre-{L}ie algebras and the rooted trees operad''.
In: \emph{Internat. Math. Res. Notices} 8 (2001), pp.~395--408.

\bibitem{ChenLiuChen}
Z. Chen, S. Liu, and L. Chen.
``Some results on anti-pre-{L}ie superalgebras and admissible {N}ovikov superalgebras''.
In: \emph{Quaest. Math.} 47.12 (2024), pp.~2447--2478.

\bibitem{DH}
V. Dotsenko and W. Heijltjes.
\emph{Operadic {G}r\"obner bases calculator}.
\url{https://irma.math.unistra.fr/~dotsenko/Operads.html}.
2014.

\bibitem{MR4549103}
V. Dotsenko, N. Ismailov, and U. Umirbaev.
``Polynomial identities in {N}ovikov algebras''.
In: \emph{Math. Z.} 303.3 (2023), Paper No. 60, 17.

\bibitem{DotsenkoKhoroshkin}
V. Dotsenko and A. Khoroshkin.
``Gr\"{o}bner bases for operads''.
In: \emph{Duke Math. J.} 153.2 (2010), pp.~363--396.

\bibitem{DKQuillen}
V. Dotsenko and A. Khoroshkin.
``Quillen homology for operads via {G}r\"{o}bner bases''.
In: \emph{Doc. Math.} 18 (2013), pp.~707--747.

\bibitem{DzhumadildaevNovikovJordan}
A. S. Dzhumadil'daev.
``Novikov-{J}ordan algebras''.
In: \emph{Comm. Algebra} 30.11 (2002), pp.~5207--5240.

\bibitem{DzhumadildaevNovikov}
A. S. Dzhumadil'daev.
``Codimension growth and non-{K}oszulity of {N}ovikov operad''.
In: \emph{Comm. Algebra} 39.8 (2011), pp.~2943--2952.

\bibitem{MR3241181}
A. S. Dzhumadil'daev and N. A. Ismailov.
``{$S_n$}- and {$GL_n$}-module structures on free {N}ovikov algebras''.
In: \emph{J. Algebra} 416 (2014), pp.~287--313.

\bibitem{MR1918188}
A. S. Dzhumadil'daev and C. L{\"o}fwall.
``Trees, free right-symmetric algebras, free {N}ovikov algebras and identities''.
In: \emph{Homology Homotopy Appl.} 4.2, part 1 (2002).
The Roos Festschrift volume, 1, pp.~165--190.

\bibitem{MR26647}
S. Eilenberg.
``Extensions of general algebras''.
In: \emph{Ann. Soc. Polon. Math.} 21 (1948), pp.~125--134.

\bibitem{adend}
D. Gao, G. Liu, and C. Bai.
``Anti-dendriform algebras, new splitting of operations and {N}ovikov-type algebras''.
In: \emph{J. Algebraic Combin.} 59.3 (2024), pp.~661--696.

\bibitem{MR1358617}
E. Getzler and M. M. Kapranov.
``Cyclic operads and cyclic homology''.
In: \emph{Geometry, topology, \& physics}.
Conf. Proc. Lecture Notes Geom. Topology, IV.
Int. Press, Cambridge, MA, 1995, pp.~167--201.

\bibitem{GinzburgKapranov}
V. Ginzburg and M. M. Kapranov.
``Koszul duality for operads''.
In: \emph{Duke Math. J.} 76.1 (1994), pp.~203--272.

\bibitem{MR1913297}
V. Hinich and A. Vaintrob.
``Cyclic operads and algebra of chord diagrams''.
In: \emph{Selecta Math. (N.S.)} 8.2 (2002), pp.~237--282.

\bibitem{Hoffbeck}
E. Hoffbeck.
``A {P}oincar\'{e}-{B}irkhoff-{W}itt criterion for {K}oszul operads''.
In: \emph{Manuscripta Math.} 131.1-2 (2010), pp.~87--110.

\bibitem{MR4610291}
Y. Hong, C. Bai, and L. Guo.
``Infinite-dimensional {L}ie bialgebras via affinization of {N}ovikov bialgebras and {K}oszul duality''.
In: \emph{Comm. Math. Phys.} 401.2 (2023), pp.~2011--2049.

\bibitem{albert}
D. P. Jacobs.
``The {A}lbert Nonassociative Algebra System: A Progress Report''.
In: \emph{Proceedings of the International Symposium on Symbolic and Algebraic Computation}.
ISSAC '94.
New York, NY, USA: Association for Computing Machinery, 1994, pp.~41--44.

\bibitem{KLSX}
P. Kolesnikov, Y. Li, Y. Sheng, and N. Xu.
``Cohomology theory of {Novikov} algebras and applications''.
2026.
arXiv: \href{https://arxiv.org/abs/2609.06304}{\texttt{2609.06304}} [math.RA].

\bibitem{LiuBai}
G. Liu and C. Bai.
``Anti-pre-{L}ie algebras, {N}ovikov algebras and commutative 2-cocycles on {L}ie algebras''.
In: \emph{J. Algebra} 609 (2022), pp.~337--379.

\bibitem{NSP}
G. Liu and C. Bai.
``New splittings of operations of {P}oisson algebras and transposed {P}oisson algebras and related algebraic structures''.
In: \emph{Algebra without borders---classical and constructive nonassociative algebraic structures---foundations and applications}.
STEAM-H: Sci. Technol. Eng. Agric. Math. Health.
Springer, Cham, [2023] \copyright 2023, pp.~49--96.

\bibitem{TPA}
G. Liu and C. Bai.
``A bialgebra theory for transposed {P}oisson algebras via anti-pre-{L}ie bialgebras and anti-pre-{L}ie {P}oisson bialgebras''.
In: \emph{Commun. Contemp. Math.} 26.8 (2024), Paper No. 2350050, 49.

\bibitem{RPA}
G. Liu and C. Bai.
``A new approach to the bialgebra theory for relative {P}oisson algebras''.
In: \emph{J. Algebraic Combin.} 62.2 (2025), Paper No. 39, 44.

\bibitem{LodayVallette}
J.-L. Loday and B. Vallette.
\emph{Algebraic operads}.
Vol.~346.
Grundlehren der mathematischen Wissenschaften [Fundamental Principles of Mathematical Sciences].
Springer, Heidelberg, 2012, pp.~xxiv+634.

\bibitem{MR2884985}
L. Makar-Limanov and U. Umirbaev.
``The {F}reiheitssatz for {N}ovikov algebras''.
In: \emph{TWMS J. Pure Appl. Math.} 2.2 (2011), pp.~228--235.

\bibitem{MarklRemm}
M. Markl and E. Remm.
``Algebras with one operation including {P}oisson and other {L}ie-admissible algebras''.
In: \emph{J. Algebra} 299.1 (2006), pp.~171--189.

\bibitem{MR4688832}
F. A. Mashurov and B. K. Sartayev.
``Metabelian {L}ie and perm algebras''.
In: \emph{J. Algebra Appl.} 23.4 (2024),
Paper No. 2450065, 12.

\bibitem{Molev}
A. I. Molev.
``The algebraic structure of the {L}ie algebra of vector fields on the line''.
In: \emph{Mat. Sb. (N.S.)} 134 (176).1 (1987), pp.~82--92, 142--143.

\bibitem{Normatov}
Z. Normatov.
``Compatible anti-pre-Lie algebras''.
In: \emph{J. Algebra Appl.} (2026), Paper No. 2650220.

\bibitem{oeis}
OEIS Foundation Inc.
\emph{The {O}n-{L}ine {E}ncyclopedia of {I}nteger {S}equences}.
Published electronically at \url{https://oeis.org}.
2026.

\bibitem{Os}
J. M. Osborn.
``Novikov algebras''.
In: \emph{Nova J. Algebra Geom.} 1.1 (1992), pp.~1--13.

\bibitem{Segal}
D. Segal.
``Free left-symmetric algebras and an analogue of the {P}oincar\'{e}-{B}irkhoff-{W}itt theorem''.
In: \emph{J. Algebra} 164.3 (1994), pp.~750--772.

\bibitem{sage}
The Sage Developers.
\emph{{S}ageMath, the {S}age {M}athematics {S}oftware {S}ystem ({V}ersion 9.5)}.
\url{https://www.sagemath.org}.
2022.

\bibitem{MR2427978}
B. Vallette.
``Manin products, {K}oszul duality, {L}oday algebras and {D}eligne conjecture''.
In: \emph{J. Reine Angew. Math.} 620 (2008), pp.~105--164.

\bibitem{ZhaoLiu}
Q. Zhao and G. Liu.
``The Levi-Civita products of Leibniz algebras with nondegenerate skew-symmetric $2$-cocycles''.
In: \emph{J. Algebra Appl.} (2027), Paper No. 2750009.

\bibitem{ZSSS}
K. A. Zhevlakov, A. M. Slin'ko, I. P. Shestakov, and A. I. Shirshov.
\emph{Rings that are nearly associative}.
Vol.~104.
Pure and Applied Mathematics.
Translated from the Russian by Harry F. Smith.
New York--London: Academic Press, Inc., 1982, pp.~xi+371.

\end{thebibliography}
\end{document}